%% file: main2.tex
\documentclass{amsart}

\usepackage[normalem]{ulem}
\usepackage{amsmath,amsthm,amsfonts,amscd,amssymb, tikz-cd} 
\usepackage{verbatim}      	
\usepackage{listings} 
\usepackage{mathtools}
\usepackage{epsfig}         	
\usepackage{url}		
\usepackage{epic,latexsym}
\usepackage{graphicx}
\usepackage{xcolor}

\usepackage{lcd}
\usepackage{mathrsfs}
\usepackage[centering,text={15.5cm,22cm},
        marginparwidth=20mm,dvips]{geometry}

\usepackage{hyperref}
\hypersetup{%
pdftitle={Preprint},
pdfauthor={},
pdfkeywords={},
bookmarksnumbered,
pdfstartview={FitH},
breaklinks=true,
colorlinks=true,
urlcolor=blue,
citecolor=blue,
linktocpage=true
}%
\usepackage[hypcap=true]{caption}

\usepackage{breakurl}

\usepackage{lscape}
\usepackage{tabularx}
\usepackage{marginnote}
\usepackage[all]{xy}
\usepackage{enumitem}
\usepackage{stmaryrd}
\usepackage{tipa}
\usepackage{upgreek}
\usepackage{tikz}
\usetikzlibrary{angles}

\DeclareMathOperator{\tw}{tw}

\DeclareMathOperator{\ev}{ev}

\DeclareMathOperator{\ord}{ord}

\def\!{\mskip-\thinmuskip}
\let\?=\overline
\let\:=\colon

\theoremstyle{plain}
\input{macroEnv.tex}

  \newtheorem{conj}[thm]{Conjecture}

\newcommand{\Q}{{\mathbb Q}}

\newcommand{\Z}{{\mathbb Z}}

\newcommand{\supp}{\operatorname{supp}}

\newcommand{\nc}{\newcommand}
\nc{\fg}{\mathfrak{g}}
\nc{\fh}{\mathfrak{h}}
\nc{\seq}{\coloneqq}
\nc{\rB}{\mathrm{B}}
\nc{\rC}{\mathrm{C}}
\nc{\rX}{\mathrm{X}}
\nc{\diag}{\mathrm{diag}}
\nc{\Cc}{\mathscr{C}}
\nc{\ep}{\varepsilon}
\nc{\II}{I}
\nc{\hI}{\hat{\II}}
\nc{\cY}{\mathcal{Y}}
\nc{\cM}{\mathcal{M}}
\nc{\tC}{\widetilde{C}}
\nc{\tc}{\tilde{c}}
\nc{\cN}{\mathscr{N}}
\nc{\ptd}{{\mathsf{pt}}}
\nc{\sL}{\mathsf{L}}
\nc{\sM}{\mathsf{M}}
\nc{\Ls}{L^\sigma}
\nc{\cZ}{\mathcal{Z}}
\nc{\res}{\mathrm{res}}
\nc{\tfg}{\tilde{\fg}}
\nc{\tII}{\tilde{\II}}
\nc{\tcM}{\tilde{\cM}}
\nc{\tcY}{\tilde{\cY}}
\nc{\tL}{\tilde{L}}
\nc{\tz}{\tilde{z}}

\nc{\tA}{\tilde{A}}
\nc{\tcC}{\tilde{\cC}}
\nc{\tx}{\tilde{x}}
\nc{\ty}{\tilde{y}}
\nc{\tY}{\tilde{Y}}
\nc{\tV}{\tilde{V}}
\nc{\tJ}{\tilde{J}}
\nc{\tLP}{\widetilde{\LP}}

\nc{\tI}{\tilde{I}}

\nc{\tsigma}{\tilde{\sigma}}
\nc{\qlp}{U_q(\mathcal{L}\fg)}
\nc{\tqlp}{U_q(\mathcal{L}\tfg)}
\nc{\qlps}{U_q(\mathcal{L}\fg^\sigma)}
\nc{\tqlps}{U_q(\mathcal{L}\tfg^{\sigma})}

\input{macro.tex}

\title{Freezing operators in representation theory of quantum loop algebras}

\author{Ryo Fujita}
\address{Research Institute for Mathematical Sciences, Kyoto University, Kitashirakawa-Oiwake-cho, Sakyo, Kyoto, 606-8502, Japan}
\email{rfujita@kurims.kyoto-u.ac.jp}

\author{Fan Qin}
\address{School of Mathematical Sciences, Beijing Normal University
No. 19, XinJieKouWai St., HaiDian District, Beijing 100875, China}
\email{qin.fan.math@gmail.com}

\thanks{}

\begin{document}

        \begin{abstract}
        We prove the Hernandez conjecture \cite{Her} on the simple $(q,t)$-characters (an analog of the Kazhdan--Lusztig conjecture) for untwisted quantum loop algebras of classical type. This result is new in type $\mathrm{C}$.
        We also prove that the folding homomorphism, introduced by Hernandez \cite{Her10}, gives a dimension-preserving bijective correspondence between the finite-dimensional simple representations (in a skeletal subcategory) of untwisted quantum loop algebras of classical simply-laced type and those of the corresponding doubly-twisted quantum loop algebras. This result is new in type $\mathrm{D}$.   
        In our approach, we develop a bootstrapping method for $q$ and $(q,t)$-characters, based on the freezing operator previously introduced in the context of cluster algebras by the second named author \cite{qin2023analogs}. This method allows us to reduce statements for general simple representations in all classical types to corresponding results on core subcategories in a uniform manner.
        \end{abstract}

        \maketitle
 
	\setcounter{tocdepth}{1} \tableofcontents{}
	
\section{Introduction}	\label{intro}

\subsection{Background}
The quantum loop algebra $\qlp$ is an affinized quantum group associated with a complex finite-dimensional simple Lie algebra $\fg$ depending on a quantum parameter $q \in \C^\times$.
Originating from the study of quantum integrable systems and solvable lattice models in statistical physics, it has been intensively studied from various viewpoints for over 40 years.

One of the central topics is the study of its finite-dimensional representations. 
Even if $q$ is generic (as we assume throughout this paper), the monoidal category $\Cc_\fg$ of finite-dimensional representations of $\qlp$ is known to be non-semisimple and non-commutative.  
The highest weight type classification of simple representations in $\Cc_\fg$ has been established by Chari--Pressley \cite{CP, CP95}.
One can also have the natural analog of the notion of character of a representation $M$ in $\Cc_\fg$, called the $q$-character $\chi_q(M)$, due to Frenkel--Reshetikhin \cite{FR}, which encodes the spectral decomposition of $M$ with respect to the action of the quantum loop Cartan subalgebra of $\qlp$.
The $q$-character admits a natural interpretation as a transfer matrix in the context of statistical physics.  
Thus, it is a fundamental and important problem to compute the $q$-character $\chi_q(L)$ of a simple representation $L$ in $\Cc_\fg$ from its highest weight.
However, except for the simplest $\mathfrak{sl}_2$ case, there is no known unified closed formula for the simple $q$-characters, as far as the authors know.  

In \cite{Nakajima04}, Nakajima established a uniform algorithm to compute the simple $q$-characters when $\fg$ is of simply-laced type (i.e., type $\mathrm{ADE}$) using the geometry of quiver varieties which he had introduced. 
It is analogous to the celebrated Kazhdan--Lusztig algorithm to compute the characters of simple highest weight representations of complex semisimple Lie algebras. 
More precisely, Nakajima constructed the $t$-analog of $q$-character (the $(q,t)$-character for short) $\chi_{q,t}(L)$ for each simple representation $L$ in $\Cc_\fg$ from its highest weight by the Kazhdan--Lusztig type algorithm, with $t$ being a formal parameter, and proved the analog of the Kazhdan--Lusztig conjecture, namely the equality 
\begin{equation} \label{eq:KL_intro}
\chi_{q,t}(L)|_{t=1} = \chi_q(L)
\end{equation}
for any simple representation $L$, based on his geometric realization of $L$ in \cite{Nak01}. 
Although the similar geometric theory is not available to the other $\fg$ of non-simply-laced type (i.e., type $\mathrm{BCFG}$), Hernandez showed that a similar algebraic construction of the simple $(q,t)$-characters $\chi_{q,t}(L)$ via the Kazhdan--Lusztig type algorithm works in a uniform way.
Then, one can expect that the specialization $\chi_{q,t}(L)|_{t=1}$ gives us the simple $q$-character $\chi_q(L)$ for any simple representation $L$ in $\Cc_\fg$ for general $\fg$, as conjectured in \cite[Conjecture 7.3]{Her}, which we refer to as the Hernandez conjecture.

In recent years, there have been several remarkable developments in the study of the category $\Cc_\fg$, including the discovery of the cluster structure of $\Cc_\fg$ and the connection with the representation theory of quiver Hecke (KLR) algebras.
The former was initiated by Hernandez--Leclerc's seminal paper \cite{HL10}, with the notion of the monoidal categorifications of cluster algebras. This topic has since attracted considerable attention and led to many important results; see, for example, \cite{NakCA, HL2, KQ, HL16, qin2017triangular, BC, KKOP1, KKOP2}. The latter was initiated by a series of works of Kang--Kashiwara--Kim--Oh \cite{KKKsw1, KKKsw2, KKKOsw3, KKKOsw4}, which introduced the generalized quantum affine Schur--Weyl duality functors. These two developments are closely related.

Based on these new developments, Hernandez, Oh, Oya and the first named author recently verified  in \cite{FHOO1} the Hernandez conjecture for type $\mathrm{B}$, and the equality \eqref{eq:KL_intro} for any simple representation in the \emph{core} subcategories (certain distinguished monoidal subcategories) of $\Cc_\fg$ for general $\fg$.

\subsection{Main result}
In this paper, we give a proof of the above Hernandez conjecture, namely the equality \eqref{eq:KL_intro} for any simple representation $L$ in $\Cc_\fg$, when $\fg$ is of classical type (i.e., type $\mathrm{ABCD})$ in a uniform way, see Theorem \ref{thm:H-conj-C}.
This is new for type $\mathrm{C}$ and provides an alternative proof of the conjecture for type $\mathrm{ABD}$ without appealing to the theory of Nakajima quiver varieties.

To establish this result, we develop a method that allows us to bootstrap statements from the core subcategories to the entire category in a uniform way for all classical types. Our method is based on an operator on  $q$ and $(q,t)$-characters, called the \emph{freezing operator}. This operator was previously introduced by the second named author in the context of cluster algebras \cite{qin2023analogs}, see Remark \ref{rem:freezing-in-cluster-alg} for more details. 
(However, our proof in this paper is logically independent from the results of \cite{qin2023analogs}.)

More precisely, we use the freezing operator to relate the $q$ and $(q,t)$-characters for a quantum loop algebra $\qlp$ with those for another quantum loop algebra $\tqlp$ of higher rank. Here we assume that the Dynkin diagram of $\fg$
is a proper subdiagram of that of $\tfg$. In our setting, the freezing operator admits a natural representation-theoretic meaning: it picks up the leading constituent of the restriction of a simple representation of $\tqlp$ to $\qlp$ at the level of the $q$-character.

Let us outline our proof. 
For a given simple representation $L$ in $\Cc_\fg$, we can always find a large enough $\tfg$ and a simple representation $\tilde{L}$ belonging to a core subcategory of $\Cc_{\tfg}$ (Remark \ref{rem:core-subcat}), whose restriction to $\qlp$ has $L$ as its leading constituent.
Then, the freezing operator $\frz$ sends $\chi_{q,t}(\tilde{L})$ and $\chi_q(\tilde{L})$ to $\chi_{q,t}(L)$ and $\chi_q(L)$ respectively. 
As mentioned above, we know that the equality \eqref{eq:KL_intro} holds for $\tilde{L}$, belonging to a core subcategory of $\Cc_{\tfg}$. Let $\ev_{t=1}$ denote the evaluation of $t$ at $1$. Then we deduce the desired equality \eqref{eq:KL_intro} from the commutativity between $\frz$ and $\ev_{t=1}$:
\[ 
\xymatrix{ \chi_{q,t}(\tilde{L}) \ar@{|->}[r]^-{\frz} \ar@{|->}[d]_-{\ev_{t=1}} & \chi_{q,t}(L) \ar@{|->}^-{\ev_{t=1}}[d]  \\
\chi_{q}(\tilde{L}) \ar@{|->}[r]^-{\frz} & \chi_q(L).}
\]

It also tells us why we cannot apply our strategy to the case when $\fg$ is of exceptional type (i.e., type $\mathrm{EFG}$): we have to find another finite-dimensional simple Lie algebra $\tfg$ of arbitrarily higher rank, which is impossible if $\fg$ is of exceptional type.  

\subsection{Application to twisted quantum loop algebras}
We find another application of the above freezing technique to the representation theory of twisted quantum loop algebras $\qlps$, where $\sigma$ is a non-trivial automorphism of the Dynkin diagram of $\fg$ (of simply-laced type).
Note that the category $\Cc_\fg^\sigma$ of finite-dimensional representations of $\qlps$ is identical to that of the quantized enveloping algebras of the twisted affine Kac--Moody algebras, which naturally arise in the Kac classification \cite{Kac90} as the Langlands dual to the untwisted affine Kac--Moody algebras of non-simply-laced type.    
For a representation $M$ in the category $\Cc^\sigma_\fg$,
Hernandez \cite{Her10} defined its twisted $q$-character $\chi_q^\sigma(M)$ as an analog of Frenkel--Reshetikhin's $q$-character. 
Although there is no known construction of the twisted $(q,t)$-characters, we have instead the folding homomorphism relating the Grothendieck rings of the category $\Cc_\fg$ and $\Cc_\fg^\sigma$ due to Hernandez \cite{Her10}, via which one may expect that the simple twisted $q$-characters can be computed from the simple (untwisted) $q$-characters, as conjectured by Wang \cite[Conjecture 2.20]{Wang}.  
Indeed, such an expectation has been verified for Kirillov--Reshetikhin modules by Hernandez \cite{Her10}. 
More recently, it has been verified by Kang--Kashiwara--Kim--Oh for any simple representation in $\Cc_\fg^\sigma$ when $\fg$ is of type $\mathrm{A}$ in \cite{KKKOsw3}, and for any simple representation in the core subcategories for general $\fg$ and $\sigma$ \cite{KKKOsw4}, through the generalized quantum affine Schur--Weyl duality functor.

In this paper, by applying the freezing operator, we verify Wang's conjecture \cite[Conjecture 2.20]{Wang} when $\fg$ is of type $\mathrm{AD}$ with $\sigma$ being of order $2$, which is new in type $\mathrm{D}$, see Theorem \ref{Thm:tw-main}.
The proof is parallel to that of the Hernandez conjecture for classical type above.
Here the role of the evaluation homomorphism $\ev_{t=1}$ is played by the folding homomorphism.

After the first version of this paper appeared on arXiv, Negu\c{t} and Wang \cite{NW26} presented a proof of \cite[Conjecture 2.20]{Wang} in full generality using folding shuffle algebras. Their approach is different from our approach based on freezing operators.

\subsection{Organization} This paper is organized as follows. 
\S\ref{Sec:pre} is the recollection of some basic facts  around finite-dimensional representations of the (untwisted) quantum loop algebras, their $q$-characters, the construction of simple $(q,t)$-characters, and the Hernandez conjecture.    
\S\ref{Sec:frz} is the main part of this paper. 
We introduce the freezing operator for $q$ and $(q,t)$-characters following \cite{qin2023analogs} in \S\S\ref{Ssec:pt}--\ref{Ssec:frz}. 
After investigating the behavior of $q$ and $(q,t)$-characters under the freezing operator in \S\S\ref{Ssec:frz-chiq}--\ref{Ssec:frz-chiqt}, we apply it to obtain the first main result (a proof of the Hernandez conjecture for classical type) in \S\ref{Ssec:proofHconj}.
In \S\ref{Sec:tw}, we discuss the application to the twisted quantum loop algebras.
The second main result (= Theorem \ref{Thm:tw-main}) is proved in \S\ref{Ssec:prf-tw-main}.

\subsection*{Acknowledgments}
The authors thank Myungho Kim for answering a question on the references. 
{ They also thank Feifei Li for pointing out an error in an earlier remark, which leads to Remark \ref{rem:fold}.}
The first named author appreciates the hospitality of Beijing Normal University during his visit in June 2025, where a part of this work was done.
The research of the first named author was supported by JSPS KAKENHI Grant No.\ JP23K12955 and No.\ JP26K16962. The research of the second named author was supported by National Natural Science Foundation of China (Grant No.\ 12422102).

\section{Simple $(q,t)$-characters and Hernandez's conjecture}
\label{Sec:pre}

In this section, we fix notations and summarize some basic facts on the finite-dimensional representation theory of (untwisted) quantum loop algebras.
We recall the construction of the simple $(q,t)$-characters and state the Hernandez conjecture following \cite{Her}.   
There is no new result in this section.

\subsection{Notation}
Let $\fg$ be a complex finite-dimensional simple Lie algebra of type $\rX_n$ in the Cartan--Killing classification, where $n$ is the rank of $\fg$.
Let $r^\vee$ denote the lacing number of $\fg$, that is, $r^\vee = 1$, $2$ or $3$ depending on whether $\rX \in \{\mathrm{A,D,E}\}$, $\rX \in \{\mathrm{B,C,F}\}$ or $\rX = \mathrm{G}$.
Let $\II$ be the set of nodes of the Dynkin diagram of $\fg$.
When $r^\vee=1$, we set $d_i = 1$ for all $i \in \II$.
When $r^\vee>1$, we set $d_i = 1$ or $r^\vee$ depending on whether the $i$-th simple root is short or long.
Then, the Cartan matrix $C = (c_{ij})_{i,j \in \II}$ of $\fg$ is determined by
\[ c_{ij} = \begin{cases}
2 & \text{if $i=j$}, \\
-\lceil d_j/d_i \rceil & \text{if $i$ and $j$ are adjacent in the Dynkin diagram}, \\
0 & \text{otherwise},
\end{cases}\]
and we have $d_i c_{ij}= d_jc_{ji}$ for any $i,j \in \II$.  
We refer to $(d_{i})_{i \in \II}$ as the minimal left symmetrizer of $C$.
We write $h^\vee$ for the dual Coxeter number of $\fg$.

\begin{eg}\label{Ex:DynkinBC}
In this paper, we are mainly interested in the case when $\fg$ is of classical type, i.e., of type $\mathrm{X}_n$ with $\mathrm{X} \in \{ \mathrm{A,B,C,D}\}$.
In this case, we identify the set $\II$ with the set $\{1,2,\ldots, n \}$ so that the Dynkin diagram of $\fg$ is depicted as:  
\begin{center}
\begin{tikzpicture}
\tikzset{dynkdot/.style={circle,draw,scale=.38}}
\def\Aoffset{2}
\node (Andesc) at (7.5,\Aoffset) {if $\rX = \mathrm{A}$,};
\node[dynkdot,label={above:{$1$}}]  (A1) at (1,\Aoffset){};
\node[dynkdot,label={above:{$2$}}]  (A2) at (2,\Aoffset){};
\node[dynkdot,label={above:{$3$}}]  (A3) at (3,\Aoffset){};
\node (Adots) at (4, \Aoffset) {$\cdots$};
\node[dynkdot,label={above:{$n-1$}}]  (An-1) at (5,\Aoffset){};
\node[dynkdot,label={above:{$n$}}]  (An) at (6,\Aoffset){};
\draw[-] (A1) -- (A2) -- (A3) -- (Adots) -- (An-1) -- (An);

\def\Boffset{1}
\node (Bndesc) at (7.5,\Boffset) {if $\rX = \rB$,};
\node[dynkdot,label={above:{$1$}}]  (B1) at (1,\Boffset){};
\node[dynkdot,label={above:{$2$}}]  (B2) at (2,\Boffset){};
\node[dynkdot,label={above:{$3$}}]  (B3) at (3,\Boffset){};
\node (Bdots) at (4, \Boffset) {$\cdots$};
\node[dynkdot,label={above:{$n-1$}}]  (Bn-1) at (5,\Boffset){};
\node[dynkdot,label={above:{$n$}}]  (Bn) at (6,\Boffset){};

\draw[-] (B2) -- (B3) -- (Bdots) -- (Bn-1) -- (Bn);
\draw[-] (B1.30) -- (B2.150);
\draw[-] (B1.330) -- (B2.210);
\draw[-] (1.4,\Boffset) -- (1.6,\Boffset-.1);
\draw[-] (1.4,\Boffset) -- (1.6,\Boffset+.1);

\def\Coffset{0}
\node (Cndesc) at (7.5,\Coffset) {if $\rX=\rC$,};
\node[dynkdot,label={above:{$1$}}]  (C1) at (1,\Coffset){};
\node[dynkdot,label={above:{$2$}}]  (C2) at (2,\Coffset){};
\node[dynkdot,label={above:{$3$}}]  (C3) at (3,\Coffset){};
\node (Cdots) at (4, \Coffset) {$\cdots$};
\node[dynkdot,label={above:{$n-1$}}]  (Cn-1) at (5,\Coffset){};
\node[dynkdot,label={above:{$n$}}]  (Cn) at (6,\Coffset){};

\draw[-] (C2) -- (C3) -- (Cdots) -- (Cn-1) -- (Cn);
\draw[-] (C1.30) -- (C2.150);
\draw[-] (C1.330) -- (C2.210);
\draw[-] (1.6,\Coffset) -- (1.4,\Coffset+.1);
\draw[-] (1.6,\Coffset) -- (1.4,\Coffset-.1);

\def\Doffset{-1.5}
\node (Dndesc) at (7.5,\Doffset) {if $\rX = \mathrm{D}$.};
\node[dynkdot,label={above:{$1$}}]  (D1) at (1,\Doffset+0.5){};
\node[dynkdot,label={above:{$2$}}]  (D2) at (1,\Doffset-0.5){};
\node[dynkdot,label={above:{$3$}}]  (D3) at (2,\Doffset){};
\node[dynkdot,label={above:{$4$}}]  (D4) at (3,\Doffset){};
\node (Ddots) at (4, \Doffset) {$\cdots$};
\node[dynkdot,label={above:{$n-1$}}]  (Dn-1) at (5,\Doffset){};
\node[dynkdot,label={above:{$n$}}]  (Dn) at (6,\Doffset){};
\draw[-] (D1) -- (D3) -- (D4) -- (Ddots) -- (Dn-1) -- (Dn) (D2)--(D3);
\end{tikzpicture}
\end{center}
We have $r^\vee =1$ or $2$ depending on whether $\mathrm{X} \in \{\mathrm{A,D}\}$ or $\mathrm{X} \in \{\mathrm{B,C}\}$, and 
\[ (d_1, d_2, d_3,\ldots, d_n) = \begin{cases}
(1,1,1,\ldots,1) & \text{if $\rX \in \{\mathrm{A,D}\},$}\\
(1,2,2,\ldots, 2) & \text{if $\rX=\rB$}, \\
(2,1,1,\ldots,1) & \text{if $\rX=\rC$}.
\end{cases}\]
The dual Coxeter number $h^\vee$ of $\fg$ is
\[ h^\vee = \begin{cases}
n+1 & \text{if $\rX = \mathrm{A}$}, \\
2n-1& \text{if $\rX=\rB$}, \\
n+1 & \text{if $\rX=\rC$}, \\
2n-2 & \text{if $\rX = \mathrm{D}$}.
\end{cases} \]
\end{eg}
 
\subsection{Quantum loop algebras}
Let $q \in \C^\times$ be a non-zero complex number which is not a root of unity.
Keeping the notation from the previous subsection,
we put $q_i \seq q^{d_i}$ for each $i \in \II$.

Associated to the simple Lie algebra $\fg$ of type $\rX_n$, we have the quantum loop algebra $\qlp$.
This is a Hopf algebra over $\C$, defined as a subquotient of the Drinfeld--Jimbo quantized enveloping algebra $U_q(\widehat{\fg})$ of the corresponding untwisted affine Kac--Moody algebra of type $\rX_n^{(1)}$ in the Kac classification \cite[\S4]{Kac90}.
As an associative $\C$-algebra, it has another presentation, called Drinfeld's loop realization, in terms of the generators $\{k_{i}^{\pm 1}, h_{i,m}, x^{\pm}_{i,l} \mid i\in \II, l,m \in \Z, m\neq 0\}$.
In particular, we have a large commutative subalgebra $U_q(\mathcal{L}\fh)$ of $\qlp$ generated by $\{k_{i}^{\pm 1}, h_{i,m}\mid i\in \II, m \in \Z \setminus \{0\}\}$.

\subsection{Finite-dimensional simple modules}
We say that a $\qlp$-module is of type $\mathbf{1}$, if the action of the element $k_i$ is diagonalizable and all its eigenvalues belong to $\{q_i^l \mid l \in \Z\}$ for all $i \in \II$. 
Let $\Cc$ denote the category of finite-dimensional $\qlp$-modules of type $\mathbf{1}$.
It carries the structure of $\C$-linear rigid monoidal abelian  category. 

There is a highest weight type classification of simple modules in the category $\Cc$ due to Chari--Pressley \cite{CP, CP95}.
It tells us that there is a bijection between the simple isomorphism classes in $\Cc$ and $\II$-tuples of monic polynomials $\pi = (\pi_i(z))_{i \in \II} \in \Pi \seq (1+z\C[z])^{\II}$, called Drinfeld polynomials.
A simple module $L(\pi)$ corresponding to $\pi \in \Pi$ is characterized by the existence of a cyclic vector $v \in L(\pi)$ satisfying 
\begin{equation} \label{eq:lhw}
x_{i,l}^+ v = 0, \quad k_i v = q_i^{\lambda_i}v, \quad h_{i,m}v = \frac{q_i^m-q_i^{-m}}{m(q_i-q_i^{-1})}\sum_{r=1}^{\lambda_i} a_{i,r}^m
\end{equation}
 for any $i \in \II$ and $l, m \in \Z$ with $m \neq 0$,  where $\pi_i(z) = \prod_{r = 1}^{\lambda_i}(1-a_{i,r}z)$.
 Such a vector $v \in L(\pi)$ is unique up to scalar. 
 We call the vector $v$ an $\ell$-highest weight vector of $L(\pi)$.
 
 \subsection{The $q$-characters} \label{Ssec:qchar}
We recall the notion of $q$-character introduced by Frenkel--Reshetikhin \cite{FR}.
Consider a variable $Y_{i,a}$ for each $i \in \II$ and $a \in \C^\times$.
For each object $M \in \Cc$, its $q$-character $\chi_q(M)$ is a Laurent polynomial in the variables $Y_{i,a}$ encoding the spectral decomposition of $M$ with respect to the action of the commutative loop Cartan subalgebra $U_q(\mathcal{L}\fh)$.
More precisely, it is known by \cite[Proposition 1]{FR} that the simultaneous eigenvalues ($\ell$-weights) of the actions of mutually commuting elements $k_i$ and $h_{i,m}$ ($i \in \II, m \in \Z\setminus \{0\}$) are respectively of the form  
\[ q_i^{\lambda_i-\mu_i} \quad \text{and} \quad \frac{q_i^m-q_i^{-m}}{m(q_i-q_i^{-1})}\left(\sum_{r=1}^{\lambda_i} a_{i,r}^m-\sum_{s=1}^{\mu_i} b_{i,s}^m \right) \]
for some $\lambda_i, \mu_i \in \Z_{\ge 0}$ and $a_{i,r}, b_{i,s} \in \C^\times$.
Then, the coefficient in $\chi_q(M)$ of the Laurent monomial $\prod_{i \in \II}(\prod_{r =1}^{\lambda_i}Y_{i, a_{i,r}} \prod_{s=1}^{\mu_i}Y^{-1}_{i,b_{i,s}})$ is defined to be the dimension of the corresponding simultaneous generalized eigenspace in $M$. 

Monomials in the variables $Y_{i,a}$ are called dominant monomials. 
By definition, a dominant monomial does not contain negative powers of $Y_{i,a}$. 
We have the natural isomorphism of multiplicative monoids between the set $\Pi$ of Drinfeld polynomials and the set of dominant monomials, given by the correspondence $(1-\delta_{i,j}az)_{j \in \II} \leftrightarrow Y_{i,a}$. 
In what follows, we identify these two sets through this isomorphism. 
In particular, for a dominant monomial $m$, we write $L(m)$ for the simple module $L(\pi)$ in the previous subsection when $m$ corresponds to $\pi$. 
By the characterization of $L(m)$, the dominant monomial $m$ occurs in the $q$-character $\chi_q(L(m))$ with multiplicity one.
The simple modules of the form $L(Y_{i,a})$ for some $i \in \II$ and $a \in \C^\times$ are called fundamental modules.

Let $K(\Cc)$ denote the Grothendieck ring of the monoidal abelian category $\Cc$.
\begin{prop}[{\cite[\S3.3]{FR}}]
The map $[M] \mapsto \chi_q(M)$ gives rise to an injective ring homomorphism
\[ \chi_q \colon K(\Cc) \hookrightarrow \Z[Y_{i,a}^{\pm 1} \mid i \in \II, a \in \C^\times].\]
In particular, the ring $K(\Cc)$ is commutative.
As a commutative ring, $K(\Cc) $ is freely generated by the classes of fundamental modules.
\end{prop}

For each $i \in \II$ and $a \in \C^\times$, we consider the Laurent monomial
\begin{equation} \label{eq:defA}
A_{i,a} \seq Y_{i,aq_i^{-1}}Y_{i,aq_i} \prod_{j \in \II, c_{ji}=-1} Y_{j,a}^{-1} \prod_{j \in \II, c_{ji}=-2} Y_{j,aq^{-1}}^{-1}Y_{j,aq}^{-1} \prod_{j \in \II, c_{ji}=-3}Y_{j,aq^{-2}}^{-1}Y_{j,a}^{-1}Y_{j,aq^2}^{-1}.\end{equation}
This is an analog of the $i$-th simple root.
For Laurent monomials $m, m'$ in the variables $Y_{i,a}$, we write $m' \le m$ if $m'm^{-1}$ is a monomial in these $A_{i,a}^{-1}$'s. 
This defines a partial ordering among Laurent monomials, called the Nakajima ordering.

\begin{thm}[\cite{FM}] \label{thm:FM}
The following statements hold.
\begin{itemize}
\item[$(1)$] The image of the $q$-character map is equal to the subring
\[ \bigcap_{i \in \II} \Z[Y_{i,a}(1+A_{i,aq_i}^{-1}), Y^{\pm 1}_{j,a} \mid j \in \II \setminus \{i\}, a \in \C^\times ].\] 
\item[$(2)$] Let $m$ be a dominant monomial. If a Laurent monomial $m'$ occurs in $\chi_q(L(m)) - m$, we have $m' < m$ with respect to the Nakajima ordering.  
\end{itemize}
\end{thm}
 
\subsection{Skeletal subcategory $\Cc_\Z$} \label{Ssec:HLcat}
Consider the subring 
\[ \cY \seq \Z[Y_{i,q^p}^{\pm 1} \mid i \in \II, p \in \Z ]\]
of $\Z[Y_{i,a}^{\pm 1} \mid i \in \II, a \in \C^\times]$.
Let $\cM$ denote the subset of $\cY$ consisting of all the Laurent monomials and $\cM_{+}$ the subset of $\cM$ consisting of dominant monomials. 
Note that $A_{i,q^p}$ belongs to $\cM$ for any $(i,p) \in \II \times \Z$.
For Laurent monomials $m, m' \in \cM$, we have $m' \le m$ with respect to the Nakajima ordering if and only if $m'm^{-1}$ is a monomial in these $A_{i,q^p}^{-1}$'s. 

We define the category $\Cc_{\Z}$ to be the Serre subcategory of $\Cc$ generated by the simple modules $L(m)$ with $m \in \cM_+$.
The category $\Cc_\Z$ is closed under tensor product, and can be thought of as a monoidal skeleton of the category $\Cc$.
In fact, any simple module in $\Cc$ factorizes into a tensor product of some spectral parameter shifts of simple modules in $\Cc_\Z$. 
(This follows from the argument in \cite[\S3.7]{HL10}. Note that our category $\Cc_\Z$ is slightly larger than the category $\Cc_\Z$ considered therein.)

For the sake of simplicity, in what follows, we set
\[ Y_{i,p} \seq Y_{i,q^p}, \qquad A_{i,p} \seq A_{i,q^p}\]
for each $(i,p) \in \II \times \Z$.
Although this is an abuse of notation, it should cause no confusion thanks to the following fact.

\begin{prop} \label{prop:FM_CZ}
The $q$-character map restricts to an injective ring homomorphism
\[ \chi_q \colon K(\Cc_\Z) \hookrightarrow \cY\]
and its image is equal $\bigcap_{i \in \II} K_i$, where $K_i$ is the subring of $\cY$ generated by the subset
\begin{equation} \label{eq:Ki}
\{Y_{i,p}(1+A_{i,p+d_i}^{-1}) \mid p \in \Z\} \cup \{Y^{\pm 1}_{j,s} \mid j \in \II\setminus\{i\}, s \in \Z  \}. 
\end{equation}
\end{prop}
\begin{proof}
It follows from Theorem~\ref{thm:FM} and \cite[Proposition 5.8]{HL10}. 
\end{proof}

\subsection{Quantum Grothendieck rings}\label{Ssec:QGR}
In this subsection, we recall the construction of quantum Grothendieck ring $K_t(\Cc_\Z)$, a deformation of the Grothendieck ring $K(\Cc_\Z)$ introduced by \cite{Nakajima04, VVq, Her}.
We follow the formulation of \cite{Her}.

Recall that $C=(c_{ij})_{i,j \in I}$ denotes the Cartan matrix of our simple Lie algebra $\fg$.
Following \cite{FR}, consider the deformed Cartan matrix $C(z) = (C_{ij}(z))_{i,j \in \II}$ given by
\[ C_{ij}(z) \seq \begin{cases}
z^{d_i} +z^{-d_i} & \text{if $i=j$}, \\
(z^{c_{ij}}-z^{-c_{ij}})/(z-z^{-1}) & \text{if $i\neq j$}.
\end{cases}\] 
Having a non-zero determinant, it belongs to $GL_{\II}(\Q(z))$.
Let $C'(z) \seq C(z)^{-1}$ denote its inverse and 
\[ 
C'_{ij}(z) = \sum_{u \in \Z} c'_{ij}(u) z^u   \quad \in \Z[\![z]\!] 
\] 
the Taylor expansion of the $(i,j)$-entry $C'_{ij}(z)$ of $C'(z)$ at $z=0$. 
All the coefficients $c'_{ij}(u)$ ($i,j \in I, u \in \Z$) are known to be integers.

Recall the multiplicative abelian group $\cM \subset \cY$ of Laurent monomials from the previous subsection.
We define a skew-symmetric bilinear map $\gamma \colon \cM \times \cM \to \Z$ by $\gamma(Y_{i,p}, Y_{j,s}) = \gamma_{ij}(p-s)$ for any $(i,p), (j,s) \in \II \times \Z$, where we set
\[ \gamma_{ij}(u) \seq c'_{ij}(u-d_i) - c'_{ij}(u+d_i) - c'_{ij}(-u-d_i) + c'_{ij}(-u+d_i).\]

Let $t$ be an indeterminate with a formal square root $t^{1/2}$. 
Consider the free $\Z[t^{\pm 1/2}]$-module 
\[\cY_t \seq \Z[t^{\pm 1/2}] \otimes_\Z \cY. \]
In what follows, we often identify $\cY$ with the subgroup $1 \otimes \cY$ of $\cY_t$, and write $y$ instead of $1 \otimes y$ for $y \in \cY$. 
The set $\cM = 1 \otimes \cM$ is a free $\Z[t^{\pm 1/2}]$-basis of $\cY_t$. 
We endow $\cY_t$ with an associative $\Z[t^{\pm \hf}]$-bilinear multiplication $*$ by
\[ m * m' \seq t^{\gamma(m,m')/2}mm'\]
for any $m,m' \in \cM$.
The resulting $\Z[t^{\pm1/2}]$-algebra $(\cY_t, *)$ is a quantum torus algebra deforming $\cY$.
For future use, we note the following formulas:
\begin{equation}\label{eq:gamma}
\gamma(A_{i,p}, A_{j,s}) = 2(\delta_{p-s, d_ic_{ij} } - \delta_{p-s,-d_ic_{ij} }), \qquad 
\gamma(A_{i,p}, Y_{j,s}) = 2\delta_{i,j}(\delta_{p-s,d_i} - \delta_{p-s,-d_i})
\end{equation}
for any $(i,p), (j,s) \in \II \times \Z$ (cf.\ \cite[Proposition 3.12]{Her} or \cite[Proposition 4.7]{HO}).

The specialization $t^{1/2} \mapsto 1$ gives rise to the ring homomorphism $\ev_{t=1} \colon (\cY_t, *) \to \cY$. 
We define the bar-involution of $(\cY_t, *)$ to be the anti-involution $y \mapsto \overline{y}$ satisfying $\overline{t^{1/2}} = t^{-1/2}$ and $\overline{m} = m$ for any $m \in \cM$. (In the literature, an element $m \in \cM$ viewed as a bar-invariant element of $\cY_t$ is sometimes referred to as a commutative monomial.)
  
 For each $i \in \II$, let $K_{i,t}$ be the $\Z[t^{\pm 1/2}]$-subalgebra of the quantum torus $(\cY_t, *)$ generated by the subset \eqref{eq:Ki}. 
 Following Hernandez \cite{Her}, we define the quantum Grothendieck ring $K_t(\Cc_\Z)$ of the category $\Cc_\Z$ to be
\[ K_t(\Cc_\Z) \seq \bigcap_{i \in \II} K_{i,t}. \] 
This is a $\Z[t^{\pm 1/2}]$-subalgebra of the quantum torus $(\cY_t, *)$, stable under the bar-involution. 
By Proposition \ref{prop:FM_CZ}, we have the following commutative diagram:
\[ \xymatrix{ 
K_t(\Cc_\Z) \ar@{^{(}->}[r] \ar[d]_-{\ev_{t=1}} & (\cY_t,*) \ar[d]^-{\ev_{t=1}}\\
\chi_q(K(\Cc_\Z)) \ar@{^{(}->}[r] & \cY,
}\]  
where the horizontal hooked arrows denote the inclusions.

\subsection{Simple $(q,t)$-characters}\label{Ssec:chiqt}
In this subsection, we recall the construction of the $(q,t)$-characters of simple representations in the category $\Cc_\Z$ following \cite{Her}.

\begin{thm}[{\cite[Theorem 5.11]{Her}, see also \cite[Remark 7.2]{HO}}]\label{Thm:F_t}
For each dominant monomial $m \in \cM_+$, there exists a unique element $F_t(m) \in K_t(\Cc_\Z)$ of the form
\begin{equation} 
F_t(m) = m + \sum_{m' \in \cM \setminus \cM_+, m' < m} s_m(m') m' 
\end{equation}
with $s_m(m') \in \Z[t^{\pm 1}]$.
Moreover, the set $\{ F_t(m) \mid m \in \cM_+\}$ is a free $\Z[t^{\pm 1/2}]$-basis of $K_t(\Cc_\Z)$.
\end{thm}

The element $F_t(m)$ is constructed through an algorithm considered in \cite[\S5.7]{Her} (a $t$-analog of the Frenkel--Mukhin algorithm \cite{FM}), whose details are recalled later in \S\ref{Ssec:tFM}. 

For a dominant monomial $m \in \cM_+$, we choose an ordered factorization $m = Y_{i_1,p_1} Y_{i_2,p_2} \cdots Y_{i_l,p_l}$ satisfying $p_1 \ge p_2 \ge \cdots \ge p_l$.
Letting 
\[\gamma(m) \seq -\frac{1}{2}\sum_{1 \le a < b \le l}\gamma(Y_{i_a,p_a}, Y_{i_b, p_b}),\] 
we set 
\begin{equation}\label{eq:def-E_t}
E_t(m) \seq t^{\gamma(m)} F_t(Y_{i_1,p_1}) * F_t(Y_{i_2,p_2})*\cdots * F_t(Y_{i_l,p_l}). 
\end{equation}
Since we have $F_t(Y_{i,p}) * F_t(Y_{j,p}) = F_t(Y_{i,p}Y_{j,p}) = F_t(Y_{j,p}) * F_t(Y_{i,p})$ for any $i,j \in \II$ and $p \in \Z$, the element $E_t(m)$ is independent of the choice of ordered factorization of $m$.
By construction, $E_t(m) - m$ is a $\Z[t^{\pm 1}]$-linear combination of monomials $m' \in \cM$ satisfying $m' < m$. 
The set $\{ E_t(m) \mid m \in \cM^+\}$ forms a free $\Z[t^{\pm 1/2}]$-basis of $K_t(\Cc_\Z)$, called the standard basis.
Applying the Kazhdan--Lusztig type algorithm, we can construct the canonical basis from the standard basis as follows.

\begin{prop}[{\cite{Nakajima04, Her}}]\label{Prop:chiqt}
For each dominant monomial $m \in \cM_+$, there exists a unique element $\chi_{q,t}(L(m))$ of $K_t(\Cc_\Z)$ satisfying 
\[ \overline{\chi_{q,t}(L(m))} = \chi_{q,t}(L(m)) \quad \text{and} \quad \chi_{q,t}(L(m)) - E_t(m) \in \sum_{m'\in \cM_+, m' < m } t^{-1}\Z[t^{-1}]E_t(m').\]
\end{prop}

The unique element $\chi_{q,t}(L(m))$ is called the $t$-analog of $q$-character, or $(q,t)$-character, of the simple representation $L(m)$.
We refer to the free $\Z[t^{\pm 1/2}]$-basis $\{ \chi_{q,t}(L(m)) \mid m \in \cM_+\}$ formed by these simple $(q,t)$-characters as the canonical basis of the quantum Grothendieck ring $K_t(\Cc_\Z)$. 

\begin{rem}\label{rem:chiqt-fund}
We have
\[
\chi_{q,t}(L(Y_{i,p})) = F_t(Y_{i,p}) = E_t(Y_{i,p})
\]
for any $(i,p) \in \II \times \Z$, as the variable $Y_{i,a}$ is minimal in $\cM_+$ with respect to the Nakajima ordering. 
\end{rem}

Using the geometry of quiver varieties, Nakajima proved the following fundamental result.

\begin{thm}[{\cite[Theorem 8.1 (2)]{Nakajima04}}] \label{thm:Nak}
When $\fg$ is of type $\mathrm{ADE}$, for any $m \in \cM_+$, we have
\begin{equation} \label{eq:KL}
\ev_{t=1}\chi_{q,t}(L(m)) = \chi_{q}(L(m)).
\end{equation}
\end{thm}

\subsection{Hernandez's conjecture}\label{Ssec:Hconj}

Having Theorem \ref{thm:Nak} above for $\fg$ of type $\mathrm{ADE}$, it is natural to expect that the same statement holds for general $\fg$, as conjectured in \cite{Her}. 
 
\begin{conj}[{\cite[Conjecture 7.3]{Her}}]  \label{conj:Her}
For a general simple Lie algebra $\fg$, the equality \eqref{eq:KL} holds for any $m \in \cM_+$.
\end{conj}

As far as the authors know, this conjecture is still open in full generality at the moment. 
Recently, we had some progress toward the conjecture, one of which is the following result.  

\begin{thm}[{\cite[Corollary 11.4 (2)]{FHOO1}}] \label{thm:H-conj-B}
The equality \eqref{eq:KL} holds for any $m \in \cM_+$ when $\fg$ is of type $\mathrm{B}$.
\end{thm}

Another partial result is the following, which will be a key ingredient in this paper.
Let $a,b$ be two integers satisfying $a \le b$, and consider the finite integer interval $[a,b] \seq \{k \in \Z \mid a \le k \le b\}$.
Define $\cM[a,b]_+$ to be the subset of $\cM_+$ consisting of the monomials in the variables $Y_{i,p}$ with $(i, p) \in \II \times [a,b]$.
Recall that $r^\vee$ and $h^\vee$ denote the lacing number of $\fg$ and the dual Coxeter number of $\fg$ respectively. 

\begin{thm}[{\cite[Corollary 9.8 (2)]{FHOO1}}]\label{thm:H-conj-region-Q}
The equality \eqref{eq:KL} holds if $m \in \cM[a,b]_+$ for some finite integer interval $[a,b] \subset \Z$ satisfying $b-a < r^\vee h^\vee$. 
\end{thm} 

\begin{rem}\label{rem:core-subcat}
The Serre subcategory $\Cc_{[a,b]}$ of $\Cc_\Z$ generated by the simple representations $L(m)$ with $ m \in \cM[a,b]_+$ is closed under tensor product, and hence a monoidal subcategory of $\Cc_\Z$. 
The category $\Cc_{[a,a+r^\vee h^\vee-1]}$ contains two subcategories $\cC_1$, $\cC_2$, which are specicial cases of the subcategory $\Cc_{\mathcal{Q}}$ associated with a Q-datum $\mathcal{Q}$ for $\fg$ (often referred to as a core subcategory or heart subcategory). See \cite[Lemma 4.17 and \S5]{FHOO1}. Each simple module $L(m)$ in $\Cc_{[a,a+r^\vee h^\vee-1]}$ is a tensor product of the form $L(m_1)\otimes L(m_2)\simeq L(m_2)\otimes L(m_1)$, where $L(m_1),L(m_2)$ are simple modules in $\cC_1,\cC_2$, respectively. The above Theorem \ref{thm:H-conj-region-Q} relies on a categorical equivalence between $\Cc_{\mathcal{Q}}$ and the monoidal category of finite-dimensional modules over quiver Hecke algebras through the generalized quantum affine Schur--Weyl duality functor \cite{KKKsw2, KO19,  OS, Fgeom, Naoi}, together with the identification of the isomorphism classes of simple modules over the quiver Hecke algebra with the elements of Lusztig's dual canonical bases due to \cite{VV, Rou}, which in particular requires a geometric argument.
\end{rem}

The first main result of this paper is the following.

\begin{thm} \label{thm:} \label{thm:H-conj-C}
The equality \eqref{eq:KL} holds true for any dominant monomial $m \in \cM_+$ when $\fg$ is of classical type, i.e., when $\mathrm{X} \in \{\mathrm{A, B, C, D}\}$.
In particular, Conjecture \ref{conj:Her} is newly confirmed to be true when $\fg$ is of type $\mathrm{C}$.
\end{thm}  
A proof will be given in the next section (see \S\ref{Ssec:proofHconj}). 
Our proof provides a new alternative proof of Theorem \ref{thm:Nak} for type $\mathrm{AD}$ without appealing to the theory of Nakajima quiver varieties, and an alternative proof of Theorem \ref{thm:H-conj-B} for type $\mathrm{B}$.

\section{Freezing operator}\label{Sec:frz}

In this section, we give a proof of our first main result (= Theorem \ref{thm:H-conj-C}).
To this end, we first introduce the freezing operator as a certain operator on the pointed elements (in the sense of \S\ref{Ssec:pt}) of the (quantum) torus algebras in \S\ref{Ssec:frz}.  
Our definition of freezing operator is an interpretation (through \cite{HL16}) of the freezing operator introduced by the second named author in \cite{qin2023analogs} in the general context of cluster algebras (see Remark \ref{rem:freezing-in-cluster-alg}).  
We investigate the behavior of $q$ and $(q,t)$-characters under the freezing operator in \S\S\ref{Ssec:frz-chiq}--\ref{Ssec:frz-chiqt} (while one technical complement is postponed in the last \S\ref{Ssec:tFM}). 
The proof of Theorem \ref{thm:H-conj-C} is given in \S\ref{Ssec:proofHconj}.

\subsection{Pointed elements}\label{Ssec:pt}
Let $\fg$ be a finite-dimensional complex simple Lie algebra.
We retain the notations of the previous section.
Following the terminology in the theory of cluster algebras \cite{qin2017triangular}, 
we say that an element $y$ of $\cY_t$ (resp.\ $\cY$) is $m$-pointed, for $m \in \cM$, if it is written in the form
\begin{equation}\label{eq:w-pointed} 
y= m + \sum_{m' \in \cM, m' < m} c_{m'}m'
\end{equation}
for some $c_{m'} \in \Z[t^{\pm 1/2}]$ (resp.\ $c_{m'} \in \Z$).
For a dominant monomial $m \in \cM_+$, the elements $F_t(m)$, $E_t(m)$,  $\chi_{q,t}(L(m))$ of $K_t(\Cc_\Z)$ are all $m$-pointed by Theorem \ref{Thm:F_t}, as well as the $q$-characters $\chi_q(L(m)) \in \cY$ by Theorem \ref{thm:FM} (2).  
An element of $\cY_t$ or $\cY$ is said to be pointed if it is $m$-pointed for some $m \in \cM$.
We denote by $\cY_t^\ptd$ (resp.\ $\cY^\ptd$) the set of pointed elements of $\cY_t$ (resp.\ $\cY$).
It is clear that the evaluation homomorphism $\ev_{t=1} \colon \cY_t \to \cY$ sends $m$-pointed elements to $m$-pointed elements for each $m \in \cM$.
In particular, it restricts to the map 
\[\ev_{t=1} \colon \cY_t^\ptd \to \cY^\ptd.\] 

\begin{rem} \label{rem:pointed}
The set $\cY_t^\ptd$ (resp.\ $\cY^\ptd$) is not a subalgebra of $\cY_t$ (resp.\ $\cY$) as the sum of two pointed elements is not a pointed element in general.
The set $\cY^\ptd$ is closed under the multiplication in $\cY$. 
If $y_i \in \cY_t$ is $m_i$-pointed ($m \in \cM$, $i = 1,2$), the element $t^{-\gamma(m_1,m_2)/2} y_1 * y_2$ is $(m_1m_2)$-pointed.  
\end{rem}

\subsection{Freezing operator}\label{Ssec:frz}
Let $\tfg$ be another finite-dimensional complex simple Lie algebra.
We denote by $(\tilde{c}_{ij})_{i,j \in \tII}$ and $(\tilde{d}_i)_{i \in \tII}$, the Cartan matrix of $\tfg$ and its minimal symmetrizer.
In what follows, we assume that there is an inclusion $\II \subset \tII$ satisfying the condition:
\begin{equation}\label{eq:I-incl} 
\tilde{c}_{ij} = c_{ij} \quad \text{and} \quad  \tilde{d}_{i} = d_i  \quad \text{for any $i,j \in \II$}.
\end{equation}

\begin{eg}\label{eg:incl}
Let $\fg$ be of type $\mathrm{X}_n$ and $\tfg$ of type $\mathrm{X}_{\tilde{n}}$ with $\mathrm{X} \in \{\mathrm{A,B,C,D}\}$ and $n \le \tilde{n}$. 
When we label the nodes of Dynkin diagrams as in Example \ref{Ex:DynkinBC}, we have the inclusion $\II = [1,n] \subset \tII=[1,\tilde{n}]$ satisfying the above assumption \eqref{eq:I-incl}.  
\end{eg}
For the same-type inclusions used in Example \ref{eg:incl}, the lacing number of $\tfg$ is the same as the lacing number $r^\vee$ of $\fg$.

We set $\tcY \seq \Z[\tilde{Y}_{i,p}^{\pm 1} \mid i \in \tII, p \in \Z]$ and write $\tcM$ (resp.\ $\tcM_+$) for its subset of Laurent monomials (resp.\ dominant monomials).   
For each $(i,p) \in \tII \times \Z$, let $\tilde{A}_{i,p}$ denote the analog of $A_{i,p}$ for $\tfg$ defined as in \eqref{eq:defA} with $(\tilde{c}_{ij})_{i,j \in \tII}$ and $(\tilde{d}_i)_{i \in \tII}$. 
We define the Nakajima partial ordering $\le$ of the set $\tcM$ with these $\tilde{A}_{i,p}$. 
We also define another partial ordering $\le_\II$ of $\tcM$ associated with the inclusion $\II \subset \tII$ by the following rule: For $m,m' \in \tcM$, we have $m' \le_\II m$ if and only if $m'm^{-1}$ is a monomial in the variables $\tilde{A}^{-1}_{i,p}$ with $(i,p) \in \II \times \Z$.
Note that $m' \le_\II m$ implies $m' \le m$, but the converse is not true in general (unless $\II = \tII$).
\begin{lem}\label{Lem:NakordI}
Let $m_1, m_2, m_3, m_4 \in \tcM$. 
\begin{itemize}
\item[$(1)$] Assuming $m_1 \le m_2 \le m_3$, we have $m_1 \le_\II m_3$ if and only if we have $m_1 \le_\II m_2 \le_\II m_3$.  
\item[$(2)$] Assuming $m_1 \le m_2$ and $m_3 \le m_4$, we have $m_1m_3 \le_\II m_2m_4$ if and only if we have both $m_1 \le_\II m_2$ and $m_3 \le_\II m_4$. 
\end{itemize} 
\end{lem}
\begin{proof}
These are immediate from the definition of $\le_\II$ and the fact that the set $\{\tilde{A}_{i,p}^{-1} \mid (i,p) \in \tII \times \Z\}$ is algebraically independent.  
\end{proof}

Consider the ring homomorphism $\res_\II \colon \tcY \to \cY$ given by 
\[ \res_{\II}(\tilde{Y}_{i,p}) \seq \begin{cases}
Y_{i,p} & \text {if $i \in \II$},\\
1 & \text{if $i \not \in \II$}
\end{cases} \]
for $(i,p) \in \tII \times \Z$.
It restricts to the maps $\tcM \to \cM$ and $\tcM_+ \to \cM_+$.
We have $\res_{\II}(\tilde{A}_{i,p}) = A_{i,p}$ for any $(i,p) \in \II \times \Z$. 
However, when $i \in \tII\setminus \II$, it is not necessarily true that we have $\res_\II(\tilde{A}_{i,p}) = 1$.
In particular, for $m,m' \in \tcM$, $m\le m'$ does not imply $\res_\II(m) \le\res_\II(m')$ in general.
On the other hand, the following implication is true for $m,m' \in \tcM$:
\begin{equation} \label{eq:imply}
m \le_\II m' \quad  \Longrightarrow \quad \res_\II(m) \le \res_\II(m').
\end{equation} 

Let $\tcY_t \seq \bigoplus_{m \in \tcM}\Z[t^{\pm 1/2}]m$ equipped with the deformed product $*$ defined similarly for $\cY_t$ but using the deformed Cartan matrix of $\tfg$.    
We extend linearly the map $\res_\II \colon \tcM \to \cM$ to the $\Z[t^{\pm 1/2}]$-linear map $\res_\II \colon \tcY_t \to \cY_t$.
Note that $\res_\II$ does not intertwine the deformed products in general.   

\begin{lem}\label{Lem:res}
For any $m_i,m_i' \in \tcM$ satisfying $m_i' \le_\II m_i$ $(i=1,2)$, we have
\[ t^{-\tilde{\gamma}(m_1, m_2)/2}\res_\II(m'_1 * m'_2) = t^{-\gamma(\res_\II(m_1),\res_\II(m_2))/2} \res_\II(m_1') * \res_\II(m_2'),\]
where $\tilde{\gamma} \colon \tcM \times \tcM \to \Z$ is the $\tfg$-analog of $\gamma$. 
\end{lem}
\begin{proof}
This is a consequence of the formulas \eqref{eq:gamma} and our assumption \eqref{eq:I-incl}.
\end{proof}
Now we are ready to define the freezing operator. 
Similarly to \cite[Definition 3.5]{qin2023analogs}, for any $m$-pointed element $y$ of $\tcY_t$ (resp.\ $\tcY$) written as
\[ y= m + \sum_{m' \in \tcM, m' < m} c_{m'}m'\]
with $c_{m'} \in \Z[t^{\pm 1/2}]$ (resp.\ $c_{m'} \in \Z$), we define the element $\frz^{\tfg}_{\fg}(y)$ to be the following $\res_\II(m)$-pointed element of $\cY_t$ (resp.\ $\cY$):
\[ \frz^{\tfg}_{\fg}(y) \seq \res_\II \left( m + \sum_{m' \in \tcM, m' <_\II m} c_{m'}m' \right).\]
In other words, the element $\frz^{\fg}_{\tfg}(y)$ is obtained from $y$ discarding the terms $c_{m'}m'$ with $m' \not \le_\II m$ and then applying $\res_\II$.
Note that $\frz^{\tfg}_{\fg}(y)$ is indeed $\res_\II(m)$-pointed by \eqref{eq:imply}. 
The assignment $y \mapsto \frz^{\tfg}_{\fg}(y)$ gives rise to the maps $\frz^{\tfg}_{\fg} \colon\tcY_t^\ptd \to \cY^\ptd_t$ and $\frz^{\tfg}_{\fg} \colon\tcY^\ptd \to \cY^\ptd$, which we call the freezing operators.
By construction, the following diagram commutes:
\begin{equation} \label{eq:frzg-ev}
\vcenter{\xymatrix{ 
\tcY_t^\ptd \ar[r]^-{\frz^{\tfg}_{\fg}}\ar[d]_-{\ev_{t=1
}} & \cY^\ptd_t \ar[d]^-{\ev_{t=1}}\\
\tcY^\ptd \ar[r]^-{\frz^{\tfg}_{\fg}} & \cY^\ptd.
}}
\end{equation}

\begin{lem}\label{Lem:frz}
For $i = 1,2$, let $m_i \in \tcM$ and $y_i \in \tcY_t^{\ptd}$ an $m_i$-pointed element.
Then, $t^{-\tilde{\gamma}(m_1, m_2)/2} y_1 * y_2$ is $(m_1m_2)$-pointed and we have
\[ \frz^{\tfg}_{\fg}(t^{-\tilde{\gamma}(m_1, m_2)/2} y_1 * y_2) = t^{-\gamma(\res_\II(m_1), \res_\II(m_2))/2} \frz^{\tfg}_{\fg}(y_1) * \frz^{\tfg}_{\fg}(y_2).\]
\end{lem}
\begin{proof}
The assertion follows from Lemma \ref{Lem:NakordI} (2) and Lemma \ref{Lem:res}. See also Remark \ref{rem:pointed}.
\end{proof}

\begin{rem}\label{rem:freezing-in-repr}
Note that the freezing operator $\frz^{\tfg}_{\fg}$ removes the factors $Y_{i,p}^{\pm 1}$ for $i\in \tI\backslash I$. We have a representation-theoretic interpretation, see Proposition \ref{prop:frz-chiq} below: For any simple module $\tL(m)$ of $\tqlp$, we take the simple module $L(\res_I (m))$ of $\qlp$, which appears as the leading Jordan--H\"{o}lder constituent of the restriction of the $\tqlp$-module $\tL(m)$ to the subalgebra $\qlp \subset \tqlp$. Then $\frz^{\tfg}_{\fg}$ sends the $q$-character of $\tL(m)$ to that of $L(\res_I (m))$. In particular, $Y_{i,p}^{\pm1}$ for $i\in\tI\backslash I$ are removed because they do not appear in the $\ell$-weight of $L(\res_I (m))$.
\end{rem}

\begin{rem}[Relation with cluster algebra freezing]\label{rem:freezing-in-cluster-alg}
Followig \cite[Theorem 5.1]{HL16}, fix one of the two connected components of the standard Hernandez--Leclerc semi-infinite 
quiver associated with
$\tilde{\mathfrak g}$, and let $\tilde{V}  \subset \tI \times \Z_{\le 0}$ be its vertex set. Let $\tilde\kappa$ be the corresponding isomorphism between quantum tori, sending the initial $y$-variables $y_{i,p}$ in the cluster algebra at $(i,p)\in
\tilde{V}$ to $\tilde A^{-1}_{i,p- d_i}$.
We refer the reader to \cite[Theorem 9.1.3]{qin2017triangular} (or \cite[Section 5.4]{FHOO23}, \cite[Section 5.1]{qin2024infinite}) for more details.

First, choose
\(
F=\{(i,p)\in\tilde {V}\mid i\notin I\}.
\)
Take any pointed Laurent polynomial $\tz$ in the initial quantum torus, and let $m$ be the leading monomial of $\tilde{\kappa}\tz$. The cluster freezing operator $\frz_F$ in \cite[Definition 3.5]{qin2023analogs} sends $\tz$ to the pointed Laurent polynomial $\frz_F\tz$. Note that $\frz_F\tz$ keeps exactly
those terms of $\tz$ whose quotients by its leading term involve no $y$-variables indexed by $F$. Via $\tilde{\kappa}$, these are precisely the terms of $\tilde{\kappa}(\tz)$ whose quotient by the leading monomial $m$ contains no factor $\tilde A^{-1}_{i,s}$ with
$i\notin I$. 

After the first step, the remaining unfrozen vertices of the quiver are precisely those $(i,p)
\in\tilde{V}$ with $i\in I$. To recover $\frz_{\mathfrak g}^{\tilde{\mathfrak g}}$ in this case, we need to further apply $\res_\II$ to $\tilde{\kappa}\frz_F(\tz)$. Then $\res_{\II}\tilde{\kappa}\frz_F(z)$ coincides with $\frz_{\mathfrak g}^{\tilde{\mathfrak g}}(\tilde{\kappa}\tz)$.
\end{rem}

\subsection{Freezing $q$-characters}\label{Ssec:frz-chiq}
For a dominant monomial $m \in \tcM_+$ for $\tfg$, we denote by $\tL(m)$ the corresponding simple $\tqlp$-module in the category $\tilde{\Cc}_\Z$, the analog of $\Cc_\Z$ for $\tfg$. The next statement is a crucial result due to Hernandez \cite{Her08} reinterpreted in our language. 

\begin{prop}[cf.\ {\cite[Lemma 5.9]{Her08}}] \label{prop:frz-chiq}
For any $m \in \tcM_+$, we have
\begin{equation}\label{eq:frz-chiq}
\frz^{\tfg}_{\fg}(\chi_q(\tL(m))) = \chi_q(L(\res_{\II}(m))). 
\end{equation}
\end{prop}
\begin{proof}
There is a $\C$-algebra homomorphism $\qlp \to \tqlp$ sending the generators $k_{i}^{\pm 1}, h_{i,s}, x^{\pm}_{i,l}$ with $i\in \II, l,s \in \Z, s\neq 0$ of $\qlp$ to the corresponding generators of $\tqlp$, via which we can regard $\tL(m)$ as a $\qlp$-module.
Let $L \subset \tL(m)$ be the $\qlp$-submodule  generated by the $\ell$-highest weight vector of $\tL(m)$. 
Then, as discussed in the proof of \cite[Lemma 5.9]{Her08}, we have $L \cong L(\res_\II(m))$ as $\qlp$-modules. 
Moreover, by considering the classical weight decomposition of $\tL(m)$, it is clear that $L$ is the sum of all the (classical) weight spaces of $\tL(m)$ whose weight is the highest weight of $\tL(m)$ minus a sum of the simple roots indexed by $\II \subset \tII$.  
This implies $\chi_q(L) = \frz^{\tfg}_{\fg}(\chi_q(\tL(m)))$ and hence the result follows. (Note that classical weights are obtained from $\ell$-weights simply by forgetting spectral parameters, namely by setting the weight of $Y_{i,p}$ to be the $i$-th fundamental weight.)
\end{proof}

\subsection{Freezing $(q,t)$-characters}\label{Ssec:frz-chiqt}

In this subsection, we are going to establish the $t$-analog of the equality \eqref{eq:frz-chiq} as in Theorem \ref{thm:frz-chiqt} below.
For any dominant monomial $m \in \tcM$, let $\tilde{F}_t(m)$ and $\tilde{E}_t(m)$ denote the $\tfg$-analogs of $F_t(m)$ and $E_t(m)$ respectively. They are $m$-pointed elements of the quantum Grothendieck ring $K_t(\tilde{\Cc}_\Z) \subset \tcY_t$ for $\tfg$ as well as the simple $(q,t)$-character $\chi_{q,t}(\tL(m))$.

\begin{prop}\label{Prop:frz-F_t}
For any $m \in \tcM_+$, we have
\[ \frz^{\tfg}_{\fg}(\tilde{F}_t(m)) = F_t(\res_\II(m)). \]
\end{prop}

We postpone the proof of this proposition to \S\ref{Ssec:tFM} below as it is somewhat technical.
As an important special case, we get the following. 
\begin{cor}\label{Cor:frz-fund}
For any $(i,p) \in \tII \times \Z$, we have
\[ \frz^{\tfg}_{\fg}(\tilde{F}_t(\tilde{Y}_{i,p})) = \begin{cases}
F_t(Y_{i,p}) & \text{if $i \in \II$}, \\
1 & \text{if $i \in \tII \setminus \II$}.
\end{cases}\]
\end{cor}

\begin{rem}
When $\fg$ is of type $\mathrm{ABC}$, Hernandez has shown in \cite[Proposition 7.2]{Her05} (recall also Remark \ref{rem:chiqt-fund}) that the equality $\chi_{q,t}(L(Y_{i,p})) = \chi_q(L(Y_{i,p}))$ in $\cY_t$ holds for any $(i,p) \in \II \times \Z$. 
Therefore, for $\fg$ of type $\mathrm{ABC}$, Corollary \ref{Cor:frz-fund} can be deduced directly from Proposition \ref{prop:frz-chiq} (and Remark \ref{rem:chiqt-fund}), without appealing to Proposition \ref{Prop:frz-F_t}.  
\end{rem}
\begin{prop}\label{Prop:frz-E_t}
For any $m \in \tcM_+$, we have
\[ \frz^{\tfg}_{\fg}(\tilde{E}_t(m)) = E_t(\res_\II(m)). \]
\end{prop}
\begin{proof}
Going back to the definition of the standard basis elements in \eqref{eq:def-E_t}, the assertion is an immediate consequence of Lemma \ref{Lem:frz} and Corollary \ref{Cor:frz-fund}.
\end{proof}

\begin{thm} \label{thm:frz-chiqt}
For any $m \in \tcM_+$, we have
\[ \frz^{\tfg}_{\fg}(\chi_{q,t}(\tL(m))) = \chi_{q,t}(L(\res_\II(m))). \]
\end{thm}
\begin{proof}
By definition, the freezing operator $\frz^{\tfg}_\fg$ intertwines the bar-involution. 
As the simple $(q,t)$-character $\chi_{q,t}(\tL(m))$ is bar-invariant, the element $\frz^{\tfg}_{\fg}(\chi_{q,t}(\tL(m)))$ is bar-invariant.
By the other characterizing property of $\chi_{q,t}(\tL(m))$ in Proposition \ref{Prop:chiqt}, we have
\[ \chi_{q,t}(\tL(m)) = \tilde{E}_t(m) + \sum_{m' < m } Q_{m,m'}(t) \tilde{E}_t(m')\]
for some $Q_{m,m'}(t) \in t^{-1} \Z[t^{-1}]$.
Applying the freezing operator $\frz^{\tfg}_\fg$ to this equation, using Lemma \ref{Lem:NakordI}(1) and Proposition \ref{Prop:frz-E_t}, we obtain
\[ \frz^{\tfg}_{\fg}(\chi_{q,t}(\tL(m))) = E_t( \res_\II(m)) + \sum_{m' <_\II m } Q_{m,m'}(t) E_t(\res_\II(m')).
\]
By \eqref{eq:imply}, one finds that the element $\frz^{\tfg}_{\fg}(\chi_{q,t}(\tL(m)))$ satisfies the two properties in Proposition \ref{Prop:chiqt} characterizing the simple $(q,t)$-character $\chi_{q,t}(L(\res_\II(m)))$. Thus, we obtain the desired equality. 
\end{proof}

\subsection{Proof of Theorem \ref{thm:H-conj-C}} \label{Ssec:proofHconj}
Now, we are ready to prove Theorem \ref{thm:H-conj-C}. Let $\fg$ and $\tfg$ be complex finite-dimensional simple Lie algebras of type $\mathrm{X}_n$ and $\mathrm{X}_{\tilde{n}}$ respectively, with $\mathrm{X} \in \{\mathrm{A,B,C,D} \}$, $1\leq n<\tilde{n}$. We choose the inclusion $\II \subset \tII$ as in Example \ref{eg:incl} and retain the notations from the previous subsections.
We are going to apply the freezing operator $\frz^{\tfg}_\fg$ while fixing $n$ but choosing $\tilde{n}$ sufficiently large relative to a given simple module of $\qlp$; the precise bound (which depends on the simple module in question) will be given below.

Given an arbitrary dominant monomial $m \in \cM_+$, 
our goal is to show the equality
\begin{equation}\label{eq:goal}
\ev_{t=1}\chi_{q,t}(L(m)) = \chi_q(L(m)). 
\end{equation}
Choose a finite integer interval $[a,b] \subset \Z$ large enough so that $m$ belongs to $\cM[a,b]_+$.
Depending on this choice of interval $[a,b]$ (and hence on $m$), we choose the integer $\tilde{n}$ large enough so that $b-a$ is smaller than the lacing number times the dual Coxeter number $r^\vee \tilde{h}^\vee$  for $\tfg$.  
Explicitly, we choose $\tilde{n} \in \Z_{> n}$ so that we have 
\[
b-a < \begin{cases}
\tilde{n}+1 & \text{if $\mathrm{X=A}$}, \\
4\tilde{n}-2 & \text{if $\mathrm{X=B}$}, \\
2\tilde{n}+2 & \text{if $\mathrm{X=C}$}, \\
2\tilde{n}-2 & \text{if $\mathrm{X=D}$}.
\end{cases}
\]

Let $\tilde{m}\in\tcM[a,b]_+$ be the dominant monomial obtained from $m$ by replacing each factor $Y_{i,p}$ with $\tilde{Y}_{i,p}$. 
Obviously, we have $\res_\II(\tilde{m})=m$.
For the corresponding simple $\tqlp$-module $\tL(\tilde{m})$, we can apply Theorem \ref{thm:H-conj-region-Q} to get
\begin{equation} \label{eq:KL-large-n}
\ev_{t=1}\chi_{q,t}(\tL(\tilde{m})) = \chi_q(\tL(\tilde{m})).
\end{equation} 
Then, we obtain
\begin{align*}
\ev_{t=1}\chi_{q,t}(L(m)) &= \ev_{t=1}\frz^{\tfg}_{\fg}(\chi_{q,t}(\tL(\tilde{m}))) && \text{by Theorem \ref{thm:frz-chiqt} with $\res_{\II}(\tilde{m})=m$,} \\
&= \frz^{\tfg}_{\fg}(\ev_{t=1} \chi_{q,t}(\tL(\tilde{m}))) && \text{by the commutativity of \eqref{eq:frzg-ev},} \\  
&= \frz^{\tfg}_{\fg}(\chi_q(\tL(\tilde{m})) && \text{by the equality \eqref{eq:KL-large-n},} \\
&= \chi_q(L(m)) && \text{by Proposition \ref{prop:frz-chiq} with $\res_\II(\tilde{m})=m$,}
\end{align*}
which verifies the desired equality \eqref{eq:goal}.

\subsection{Proof of Proposition \ref{Prop:frz-F_t}}\label{Ssec:tFM}
We have to recall how to construct the element $F_t(m) \in K_t(\Cc_\Z)$ in Theorem \ref{Thm:F_t}. 
First, we prepare some terminology. 
For each $i \in \II$, we say that an element $m$ of $\cM$ is $i$-dominant if it does not contain $Y_{i,p}^{-1}$ for any $p \in \Z$ as its non-trivial factor. 
Let $\cM_{i,+} \subset \cM$ denote the set of $i$-dominant monomials. Note that we have $\cM_+ = \bigcap_{i \in \II}\cM_{i,+}$ by definition.
For any $m \in \cM_{i,+}$, there is a unique $m$-pointed element $F_{i,t}(m) \in K_{i,t}$ (recall the definition of $K_{i,t}$ from \S\ref{Ssec:QGR}) such that $F_{i,t}(m)=my$ with $y \in \Z[t^{\pm 1}][A_{i,p}^{-1} \mid p \in \Z]$ and $m$ is the unique $i$-dominant monomial occurring in $F_{i,t}(m)$.  
See \cite[Proposition 4.12]{Her} for the construction of $F_{i,t}(m)$. We denote by $[F_{i,t}(m)]_{m'} \in \Z[t^{\pm 1}]$ the coefficient of $m' \in \cM$ in $F_{i,t}(m)$.
Note that we have $[F_{i,t}(m)]_{m'} \neq 0$ only if $m' = mA_{i,p_1}^{-1}A_{i,p_2}^{-1} \cdots A_{i,p_k}^{-1}$ for some $k \in \Z_{\ge 0}$ and $p_1,p_2,\ldots,p_k \in \Z$.

Given a dominant monomial $m \in \cM_+$,  we consider the set $D(m)$ of Laurent monomials $m' \in \cM$ satisfying the following property: there is a finite sequence $m=m_0, m_1, \ldots, m_l = m'$ in $\cM$ such that, for each $k \in [1, l]$, there is $k' \in [0, k-1]$ such that $m_{k'}$ is $i$-dominant and $[F_{i,t}(m_{k'})]_{m_k} \neq 0$ for some $i \in \II$.  
We fix a total ordering $D(m) = \{ m=m_0, m_1, m_2, \ldots\}$ compatible with the Nakajima partial ordering, i.e., so that $m_k > m_l$ implies $k<l$. 
Then, we recursively define the sequences $(s(m_k))_{k \ge 0}$ and $(s_{i}(m_k))_{k \ge 0}$ for $i \in \II$ in $\Z[t^{\pm 1}]$ by the following procedure (a $t$-analog of Frenkel--Mukhin algorithm): 
\begin{itemize}
\item for $k=0$, we set $s(m_0)\seq 1$ and $s_{i}(m_0) \seq 0$ for all $i \in \II$;
\item having constructed up to $k$-th terms $(s(m_u))_{u \in  [0,k]}$ and $(s_i(m_u))_{u \in [0,k]}$ for all $i \in \II$, we set
\[ s_i(m_{k+1}) \seq \sum_{u \in[0,k], m_u \in \cM_{i,+}}\big(s(m_u)-s_i(m_u)\big)[F_{i,t}(m_u)]_{m_{k+1}} \qquad \text{for each $i \in \II$},\]
and then we set 
\[ s(m_{k+1}) \seq \begin{cases}
s_i(m_{k+1}) & \text{if $m_{k+1} \not \in \cM_{i,+}$ for some $i \in \II$}, \\
0 & \text{if $m_{k+1} \in \cM_+$},
\end{cases}\]
which is well-defined since we know that $m_{k+1} \not \in \cM_{i,+} \cup \cM_{j,+}$ implies $s_{i}(m_{k+1}) = s_{j}(m_{k+1})$ by \cite[Lemma 5.22]{Her}.
\end{itemize}
Then, the element $F_t(m) \seq \sum_{k \ge 0} s(m_k) m_k$ is shown to satisfy the desired property in Theorem \ref{Thm:F_t} by \cite[Lemma 5.25]{Her}.
The resulting element $F_t(m)$ does not depend on the choice of total ordering of $D(m)$.

Now we give a proof of Proposition \ref{Prop:frz-F_t}.
We retain the assumption of \S\ref{Ssec:frz}. 
In particular, the set $\II$ is included in the set $\tII$ for a larger Lie algebra $\tfg$. 
Given a dominant monomial $m \in \tcM_+$ for $\tfg$, we consider the $\tfg$-analog $\tilde{D}(m) \subset \tcM$ of the set $D(m) \subset \cM$ as above, choose a total ordering $\tilde{D}(m) = \{ m=\tilde{m}_0, \tilde{m}_1, \tilde{m}_2, \ldots\}$ compatible with the Nakajima ordering, and define the sequences $(\tilde{s}(\tilde{m}_k))_{k \ge 0}$ and $(\tilde{s}_i(\tilde{m}_k))_{k \ge 0}$ for $i \in \tII$ in $\Z[t^{\pm 1}]$ by the $\tfg$-analog of the above recursive algorithm.  
Then, we have $\tilde{F}_t(m) = \sum_{k \ge 0}\tilde{s}(\tilde{m}_k)\tilde{m}_k$. 
The freezing operator sends it to 
\[ \frz^{\tfg}_{\fg}(\tilde{F}_t(m)) = \sum_{k \ge 0, \tilde{m}_k \le_\II m} \tilde{s}(\tilde{m}_k)\res_{\II}(\tilde{m}_k) = \sum_{k \ge 0} \tilde{s}(\tilde{m}_{\nu(k)}) \res_\II(\tilde{m}_{\nu(k)}),\]
where $\nu \colon \Z_{\ge 0} \to \Z_{\ge 0}$ is the unique increasing function satisfying $\nu(0)=0$ and $\{ \tilde{m}_{\nu(k)} \mid k\ge 0\} = \{m' \in \tilde{D}(m) \mid m' \le_\II m \}$.
Observe that the set $\{ m_k \seq \res_\II(\tilde{m}_{\nu(k)}) \mid k\ge 0\} \subset \cM$ coincides with the set $D(\res_\II(m))$.
Moreover, $D(\res_\II(m)) = \{ \res_\II(m) = m_0, m_1, m_2, \ldots \}$ is a total ordering compatible with the Nakajima ordering. 
Therefore, it is enough to show that we have $\tilde{s}(\tilde{m}_{\nu(k)}) = s(m_k)$ for any $k \ge 0$. 
We prove it by induction on $k$ together with the equality $\tilde{s}_i(\tilde{m}_{\nu(k)}) = s_i(m_k)$ for all $i \in \II$.
For $k=0$, they are clear. 
Assume that we have verified the equalities for up to $k$-th terms. 
Then, for each $i \in \II$, the algorithm yields
\begin{align*} \tilde{s}_i(\tilde{m}_{\nu(k+1)}) &= \sum_{u\in [0,\nu(k+1)-1], \tilde{m}_u \in \tilde{\cM}_{i,+}}\big(\tilde{s}(\tilde{m}_u)-\tilde{s}_i(\tilde{m}_u)\big)[\tilde{F}_{i,t}(\tilde{m}_u)]_{\tilde{m}_{\nu(k+1)}} \\
&= \sum_{u\in [0,k], \tilde{m}_{\nu(u)} \in \tilde{\cM}_{i,+}}\big(\tilde{s}(\tilde{m}_{\nu(u)})-\tilde{s}_i(\tilde{m}_{\nu(u)})\big)[\tilde{F}_{i,t}(\tilde{m}_{\nu(u)})]_{\tilde{m}_{\nu(k+1)}} \\
&= \sum_{u\in [0,k], m_{\nu(u)} \in \cM_{i,+}}\big(s(m_u)-s_i(m_u)\big)[F_{i,t}(m_{u})]_{m_{k+1}} \\
&= s_i(m_{k+1}),
\end{align*}
where the second equality follows because we have $[\tilde{F}_{i,t}(\tilde{m}_u)]_{\tilde{m}_{\nu(k+1)}} = 0$ if $\tilde{m}_u \not \le_\II m$ by Lemma \ref{Lem:NakordI} (1), the third equality follows by the induction hypothesis and by the construction of $F_{i,t}(m_u)$.
Note that the condition $\tilde{m}_{\nu(k+1)} \in \tcM_{+}$ is equivalent to the condition $\tilde{m}_{\nu(k+1)} \in \bigcap_{i \in \II}\tcM_{i,+}$ because $\tilde{m}_{\nu(k+1)} \in \bigcap_{i \in \tII\setminus \II}\tcM_{i,+}$ is automatically satisfied under the assumption $\tilde{m}_{\nu(k+1)} \le_\II m$. 
Moreover, it is also equivalent to $m_{k+1}=\res_\II(\tilde{m}_{\nu(k+1)}) \in \cM_+$. Therefore, we have
\begin{align*}
\tilde{s}(\tilde{m}_{\nu(k+1)}) &= \begin{cases}
\tilde{s}_i(\tilde{m}_{\nu(k+1)}) & \text{if $\tilde{m}_{\nu(k+1)} \not\in \tcM_{i,+}$ for some $i \in \II$}, \\
0 & \text{otherwise}
\end{cases}\\
&= \begin{cases}
s_i(m_{k+1}) & \text{if $m_{k+1} \not\in \cM_{i,+}$ for some $i \in \II$}, \\
0 & \text{if $m_{k+1} \in \cM_+$}
\end{cases}\\
&= s(m_{k+1}).
\end{align*}
This completes the induction and  we obtain the desired equality 
\[\frz^{\tfg}_{\fg}(\tilde{F}_t(m)) = \sum_{k \ge 0} \tilde{s}(\tilde{m}_{\nu(k)}) \res_\II(\tilde{m}_{\nu(k)}) = \sum_{k \ge 0} s(m_k) m_k = F_t(\res_\II(m)).  \]
This proves the desired statement.

\section{Application to twisted quantum loop algebras} \label{Sec:tw}

In this section, we discuss another application of the freezing operator to the finite-dimensional representation theory of \emph{twisted} quantum loop algebras. 
After fixing some notations, we recall in \S\ref{Ssec:twqchar} the definition of twisted $q$-characters (a twisted analog of $q$-characters) introduced by Hernandez in \cite{Her10}. 
Although a twisted analog of $(q,t)$-characters has not been constructed yet in the literature, we have instead the \emph{folding homomorphism}, relating the Grothendieck rings for untwisted quantum loop algebras with those for twisted ones, again due to Hernandez \cite{Her10}, as we recall in \S\ref{Ssec:folding}.
The folding homomorphism is conjectured to give a bijective correspondence between the simple isomorphism classes in skeletal subcategories by Wang \cite[Conjecture 2.20]{Wang}.
As the second main result of the present paper, we verify this conjecture for doubly-twisted quantum loop algebras of type $\mathrm{AD}$ in \S\ref{Ssec:prf-tw-main} by applying the freezing operators.
This result is new for type $\mathrm{D}$. (It has already been known for type $\mathrm{A}$ by \cite{KKKOsw3} with a different method.)    

\subsection{Twisted quantum loop algebras} \label{Ssec:tw-setup}
Throughout this section, let $\fg$ be a complex finite-dimensional simple Lie algebra of simply-laced type, i.e., of type $\mathrm{X}_n$ with $\mathrm{X} \in \{ \mathrm{A, D, E}\}$, and $C=(c_{ij})_{i,j \in \II}$ the Cartan matrix of $\fg$.
Let $\sigma$ denote a non-trivial automorphism of the Dynkin diagram of $\fg$, i.e., a non-trivial bijection $\sigma \colon \II \to \II$ satisfying $c_{\sigma(i)\sigma(j)} = c_{ij}$ for any $i, j \in \II$.
We denote by $r$ the order of $\sigma$.
Such an automorphism $\sigma$ is one of the following five cases up to relabeling of the Dynkin nodes. Here we identify the set $\II$ with a certain finite subset of $\Z$ and the action of $\sigma$ is depicted by the blue arrows.
Note that $\sigma$ is uniquely determined by the type $\mathrm{X}_n$ and the order $r$ (up to relabeling).
\begin{itemize}
\item
$(\mathrm{X}_n,r) = (\mathrm{A}_{2l}, 2)$ with $l \in \Z_{>0}$:
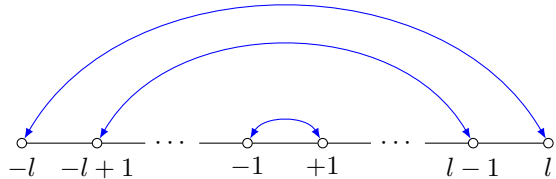
\begin{center}
\begin{tikzpicture}
\tikzset{dynkdot/.style={circle,draw,scale=.38}}
\def\Aoffset{2}
\node[dynkdot,label={below:{$-l$}}]  (A-l) at (1,\Aoffset){};
\node[dynkdot,label={below:{$-l+1$}}]  (A-l+1) at (2,\Aoffset){};
\node (A-dots) at (3, \Aoffset) {$\cdots$};
\node[dynkdot,label={below:{$-1$}}]  (A-1) at (4,\Aoffset){};
\node[dynkdot,label={below:{$+1$}}]  (A+1) at (5,\Aoffset){};
\node (A+dots) at (6, \Aoffset) {$\cdots$};
\node[dynkdot,label={below:{$l-1$}}]  (A+l-1) at (7,\Aoffset){};
\node[dynkdot,label={below:{$l$}}]  (A+l) at (8,\Aoffset){};
\draw[-] (A-l)--(A-l+1)--(A-dots)--(A-1)--(A+1)--(A+dots)--(A+l-1)--(A+l);
\coordinate (A0) at (4.5,\Aoffset); 
\draw[latex-latex, blue] (A-1) to [out=60, in=120] (A+1);
\draw[latex-latex, blue] (A-l+1) to [out=60, in=120] (A+l-1);
\draw[latex-latex, blue] (A-l) to [out=60, in=120] (A+l);
\end{tikzpicture}
\end{center}
\item $(\mathrm{X}_n,r) = (\mathrm{A}_{2l+1},2)$ with $l \in \Z_{>0}$:
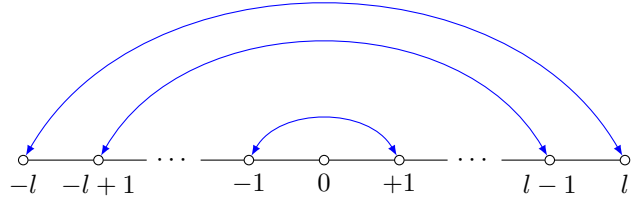
\begin{center}
\begin{tikzpicture}
\tikzset{dynkdot/.style={circle,draw,scale=.38}}
\def\Aoffset{2}
\node[dynkdot,label={below:{$-l$}}]  (A-l) at (1,\Aoffset){};
\node[dynkdot,label={below:{$-l+1$}}]  (A-l+1) at (2,\Aoffset){};
\node (A-dots) at (3, \Aoffset) {$\cdots$};
\node[dynkdot,label={below:{$-1$}}]  (A-1) at (4,\Aoffset){};
\node[dynkdot,label={below:{$0$}}]  (A0) at (5,\Aoffset){};
\node[dynkdot,label={below:{$+1$}}]  (A+1) at (6,\Aoffset){};
\node (A+dots) at (7, \Aoffset) {$\cdots$};
\node[dynkdot,label={below:{$l-1$}}]  (A+l-1) at (8,\Aoffset){};
\node[dynkdot,label={below:{$l$}}]  (A+l) at (9,\Aoffset){};
\draw[-] (A-l)--(A-l+1)--(A-dots)--(A-1)--(A0)--(A+1)--(A+dots)--(A+l-1)--(A+l);
\draw[latex-latex, blue] (A-1) to [out=60, in=120] (A+1);
\draw[latex-latex, blue] (A-l+1) to [out=60, in=120] (A+l-1);
\draw[latex-latex, blue] (A-l) to [out=60, in=120] (A+l);
\end{tikzpicture}
\end{center}
\item $(\mathrm{X}_n,r) = (\mathrm{D}_{n},2)$ with $n \in \Z_{\ge 4}$:
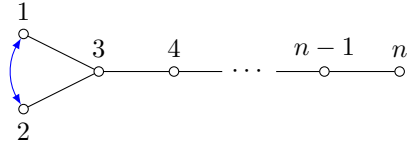
\begin{center}
\begin{tikzpicture}
\tikzset{dynkdot/.style={circle,draw,scale=.38}}
\def\Doffset{-1.5}
\node[dynkdot,label={above:{$1$}}]  (D1) at (1,\Doffset+0.5){};
\node[dynkdot,label={below:{$2$}}]  (D2) at (1,\Doffset-0.5){};
\node[dynkdot,label={above:{$3$}}]  (D3) at (2,\Doffset){};
\node[dynkdot,label={above:{$4$}}]  (D4) at (3,\Doffset){};
\node (Ddots) at (4, \Doffset) {$\cdots$};
\node[dynkdot,label={above:{$n-1$}}]  (Dn-1) at (5,\Doffset){};
\node[dynkdot,label={above:{$n$}}]  (Dn) at (6,\Doffset){};
\draw[-] (D1) -- (D3) -- (D4) -- (Ddots) -- (Dn-1) -- (Dn) (D2)--(D3);
\draw[latex-latex, blue] (D1) to [out=240, in=120] (D2);
\end{tikzpicture}
\end{center}
\item $(\mathrm{X}_n,r) = (\mathrm{D}_4,3)$:
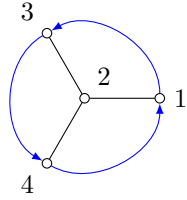
\begin{center}
\begin{tikzpicture}
\tikzset{dynkdot/.style={circle,draw,scale=.38}}
\node[dynkdot,label={right:{$1$}}]  (D1) at (0:1){};
\node[dynkdot,label={above right:{$2$}}]  (D2) at (0,0){};
\node[dynkdot,label={above left:{$3$}}]  (D3) at (120:1){};
\node[dynkdot,label={below left:{$4$}}]  (D4) at (240:1){};
\draw[-] (D2)--(D1) (D2)--(D3) (D2)--(D4);
\draw[-latex, blue] (D1) to [out=90, in=30] (D3);
\draw[-latex, blue] (D3) to[out=210, in=150] (D4);
\draw[-latex, blue] (D4) to[out=330, in=270] (D1);
\end{tikzpicture}
\end{center}
\item $(\mathrm{X}_n,r) = (\mathrm{E}_6,2)$:
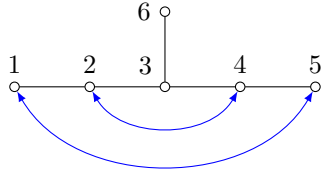
\begin{center}
\begin{tikzpicture}
\tikzset{dynkdot/.style={circle,draw,scale=.38}}
\node[dynkdot,label={above:{$1$}}]  (E1) at (-2,0){};
\node[dynkdot,label={above:{$2$}}]  (E2) at (-1,0){};
\node[dynkdot,label={above left:{$3$}}]  (E3) at (0,0){};
\node[dynkdot,label={above:{$4$}}]  (E4) at (1,0){};
\node[dynkdot,label={above:{$5$}}]  (E5) at (2,0){};
\node[dynkdot,label={left:{$6$}}]  (E6) at (0,1){};
\draw[-] (E1)--(E2)--(E3)--(E4)--(E5) (E3)--(E6);
\draw[latex-latex, blue] (E1) to [out=300, in=240] (E5);
\draw[latex-latex, blue] (E2) to[out=300, in=240] (E4);
\end{tikzpicture}
\end{center}
\end{itemize}
Note that $r \in \{2,3\}$. 
We fix a primitive $r$-th root of unity $\omega$ in $\C$.
 
Associated with the above $\fg, \sigma$ and a quantum parameter $q \in \C^\times$ which is not a root of unity, the twisted quantum loop algebra $\qlps$ is defined.
We follow the convention of \cite{Her10}.
This is a Hopf algebra over $\C$, defined as a subquotient of the Drinfeld--Jimbo quantized enveloping algebra $U_q(\widehat{\fg}^\sigma)$ of the twisted affine Kac--Moody algebra $\widehat{\fg}^\sigma$ of type $\rX_n^{(r)}$ in the Kac classification \cite{Kac90}.
It has Drinfeld-type generators $\{k_{i}^{\pm 1}, h_{i,m}, x^{\pm}_{i,l} \mid i\in \II, l,m \in \Z, m\neq 0\}$ subject to
\begin{equation}\label{eq:tw-Drinfeld}
k_{\sigma(i)}^{\pm 1} = k_{i}^{\pm 1}, \quad h_{\sigma(i),m} = \omega^m h_{i,m}, \quad x^{\pm}_{\sigma(i),l} = \omega^l x^{\pm}_{i,l} 
\end{equation}
among other relations. See \cite{Dam}.
We have a large commutative subalgebra $U_q(\mathcal{L}\fh^\sigma)$ of $\qlps$ generated by $\{k_{i}^{\pm 1}, h_{i,m}\mid i\in \II, m \in \Z \setminus \{0\}\}$, which we call the loop Cartan subalgebra.

\subsection{Finite-dimensional simple modules}
We say that a $\qlps$-module is of type $\mathbf{1}$, if the action of the element $k_i$ is diagonalizable and all its eigenvalues are integer powers of $q$ for all $i \in \II$.
Let $\Cc^\sigma$ denote the category of finite-dimensional $\qlps$-modules of type $\mathbf{1}$.
It carries the structure of $\C$-linear rigid monoidal abelian  category. 

Analogously to the untwisted case, 
there is a highest weight type classification of simple modules in $\Cc^\sigma$ due to Chari--Pressley \cite{CPtwisted} again in terms of the Drinfeld polynomials.
Let $\Pi \seq (1 + z\C[z])^{\II}$ as before, and consider the action of $\langle \sigma \rangle \cong \Z/r\Z$ on $\Pi$ given by $\sigma \cdot \pi \seq (\pi_{\sigma^{-1}(i)}(\omega z) )_{i \in \II}$
for $\pi = (\pi_i(z))_{i \in \II} \in \Pi$.
Let $\Pi^\sigma \subset \Pi$ denote the subset of $\sigma$-invariant Drinfeld polynomials, i.e., the elements $\pi = (\pi_i(z))_{i \in \II} \in \Pi$ satisfying $\pi_{\sigma(i)}(z) = \pi_i(\omega z)$ for all $i \in \II$.
The classification theorem tells us that the set of the simple isomorphism classes of $\Cc^\sigma$ is in bijection with the set $\Pi^\sigma$.
A simple module $L^\sigma(\pi)$ corresponding to $\pi \in \Pi^\sigma$ is characterized by the existence of the $\ell$-highest weight vector  $v \in \Ls(\pi)$ satisfying the same equations as in \eqref{eq:lhw} (with $q_i=q$ for all $i \in \II$). 
 Note that the relation \eqref{eq:tw-Drinfeld} forces $\pi$ to be $\sigma$-invariant.

 \subsection{Twisted  $q$-characters} \label{Ssec:twqchar}
As a twisted analog of Frenkel--Reshetikhin's $q$-character, Hernandez \cite{Her10} 
introduced the twisted $q$-character $\chi_q^\sigma(M)$ for each object $M \in \Cc^\sigma$, which is a Laurent polynomial in the variables $Y_{i,a}$, $(i,a) \in \II \times \C^\times$, encoding the spectral decomposition of $M$ with respect to the action of $U_q(\mathcal{L}\fh^\sigma)$.
Similarly to the untwisted case, it is shown in \cite[Theorem 3.2]{Her10} that simultaneous eigenvalues of the actions of mutually commuting elements $k_i$ and $h_{i,m}$, $i \in \II, m \in \Z\setminus \{0\}$, are respectively of the form $q^{\lambda_i-\mu_i}$ and $\frac{q^m-q^{-m}}{m(q-q^{-1})}(\sum_{r=1}^{\lambda_i} a_{i,r}^m-\sum_{s=1}^{\mu_i} b_{i,s}^m)$
for some $\lambda_i, \mu_i \in \Z_{\ge 0}$ and $a_{i,r}, b_{i,s}\in \C^\times$.
Then, the coefficient in $\chi_q^\sigma(M)$ of the Laurent monomial $\prod_{i \in \II}(\prod_{r =1}^{\lambda_i}Y_{i, a_{i,r}} \prod_{s=1}^{\mu_i}Y^{-1}_{i,b_{i,s}})$ is defined to be the dimension of the corresponding simultaneous generalized eigenspace in $M$.

Let $(\II \times \C^\times)/\sigma$ denote the set of orbits with respect to the action of $\langle \sigma \rangle \cong \Z/r\Z$ on $\II \times \C^\times$ given by 
\[ \sigma \cdot (i,a) \seq (\sigma(i), \omega a)\]
for $i \in \II$ and $a \in \C^\times$.
We write $[i,a]$ for the orbit of $(i,a) \in \II \times \C^\times$. 
Note that the relation \eqref{eq:tw-Drinfeld} forces that the $q$-character $\chi_q^\sigma(M)$ is a Laurent polynomial in the variables
\begin{equation} \label{eq:Z-variable}
Y_{[i,a]} \seq \prod_{k = 1}^r Y_{\sigma^k(i), \omega^ka} \qquad (i \in \II, a \in \C^\times).
\end{equation}
As the notation suggests, the element $Y_{[i,a]}$ only depends on the orbit $[i,a]$ of $(i,a)$, i.e., we have $Y_{[i,a]} = Y_{[j,b]}$ if $[i,a] = [j,b]$.
Therefore, the assignment $[i,a] \mapsto Y_{[i,a]}$ is well-defined, 
and the set $\{ Y_{[i,a]} \mid [i,a] \in (\II \times \C^\times)/\sigma\}$ is an algebraically independent subset of $\Z[Y_{i,a}^{\pm 1} \mid (i,a) \in \II \times \C^\times]$. 
Elements in the subring $\Z[Y_{[i,a]}^{\pm 1} \mid [i,a] \in (\II \times \C^\times)/\sigma]$ are said to be $\sigma$-invariant.

Recall from \S\ref{Ssec:qchar} the natural isomorphism of multiplicative monoids between the set $\Pi$ and the set of dominant monomials in the variables $Y_{i,a}$. 
Under the isomorphism, the subset $\Pi^\sigma \subset \Pi$ corresponds to the monomials in the variables $Y_{[i,a]}$, i.e., the $\sigma$-invariant dominant monomials. 
In what follows, for a $\sigma$-invariant dominant monomial $m$, we write $L^\sigma(m)$ for the simple module $L^\sigma(\pi)$ in the previous subsection when $m$ corresponds to $\pi \in \Pi^\sigma$. 
By definition, the monomial $m$ occurs in the twisted $q$-character $\chi_q^\sigma(L^\sigma(m))$ with multiplicity one.
The simple modules of the form $L^\sigma(Y_{[i,a]})$ for some $[i,a] \in (\II \times \C^\times)/\sigma$ are called fundamental modules.

\begin{rem} \label{rem:Z}
Our definition of the twisted $q$-characters is slightly different from that of \cite{Her10} as Laurent polynomials in the variables $Y_{i,a}$. 
In fact, our twisted $q$-character  is obtained from the twisted $q$-character in \cite{Her10} by replacing $Y_{i,a}$ with $\prod_{a' \in \C^\times, (a')^{r_i} = a} Y_{i,a'}$ for each $i \in \II$ and $a \in \C^\times$, where $r_i$ is the order of the stabilizer of $i$ in $\langle \sigma \rangle$.
In \cite{Her10}, it was shown that the twisted $q$-characters can be expressed as Laurent polynomials in another kind of variables $Z_{i,a}$. 
By definition, the variable $Z_{i,a}$ in \cite{Her10} equals $Y_{i,a}$ if $\sigma(i)=i$ (which is different from our $Y_{[i,a]}$) and it equals $\prod_{k = 1}^r Y_{\sigma^k(i), \omega^k a}$ if $\sigma(i) \neq i$ (which coincides with our $Y_{[i,a]}$).
As a result, our twisted $q$-characters obtained from those of \cite{Her10} simply by replacing $Z_{i,a}$ with $Y_{[i,a]}$ for all $(i,a) \in \II \times \C^\times$.
An advantage of our convention is that it enables us to describe the folding homomorphism (explained below) in a uniform way.  
\end{rem}

Let $ K(\Cc^\sigma)$ denote the Grothendieck ring of the abelian monoidal category $\Cc^\sigma$.

\begin{prop}[{\cite[\S3.2]{Her10}}]
The assignment $[M] \mapsto \chi_q^\sigma(M)$ gives rise to an injective ring homomorphism
\[ \chi_q^\sigma \colon K(\Cc^\sigma) \hookrightarrow \Z[Y_{[i,a]}^{\pm 1} \mid [i,a] \in (\II \times \C^\times)/\sigma] \subset \Z[Y_{i,a}^{\pm 1} \mid (i,a) \in \II \times \C^\times].\]
In particular, the ring $K(\Cc^\sigma)$ is commutative.
As a commutative ring, $K(\Cc^\sigma) $ is freely generated by the classes of fundamental modules.
\end{prop}

For each $(i,a) \in \II \times \C^\times$, we define
\[ A_{[i,a]} \seq Y_{[i,aq^{-1}]}Y_{[i,aq]} \prod_{j \in \II, c_{ji}=-1} Y_{[j,a]}^{-1}.\]
Again, as the notation suggests, the element $A_{[i,a]}$ only depends on the orbit $[i,a]$ of $(i,a)$, i.e., we have $A_{[i,a]} = A_{[j,b]}$ if $[i,a] = [j,b]$.
For $\sigma$-invariant Laurent monomials $m, m'$, we write $m' \le^\sigma m$ if $m'm^{-1}$ is a monomial in these $A_{[i,a]}^{-1}$'s. 
This defines a partial ordering among $\sigma$-invariant Laurent monomials.

\begin{prop}[\cite{Her10}] \label{prop:tw-domin}
Let $m$ be a $\sigma$-invariant dominant monomial. If a $\sigma$-invariant Laurent monomial $m'$ occurs in $\chi_q^\sigma(L^\sigma(m)) - m$, we have $m' <^\sigma m$. 
\end{prop}
\begin{proof}
It is enough to show the statement for fundamental modules.  
For this case, it is established in \cite[Lemma 4.14]{Her10}. 
See also Remark \ref{rem:Z} above.
\end{proof}

\subsection{Folding homomorphism} \label{Ssec:folding}
Consider the ring homomorphism
\[ \phi^\sigma_\fg \colon \Z[Y_{i,a}^{\pm 1} \mid (i,a) \in \II \times \C^\times ] \to \Z[Y_{[i,a]}^{\pm 1} \mid [i,a] \in (\II \times \C^\times)/\sigma] \]
given by $\phi^\sigma_\fg(Y_{i,a}) \seq  Y_{[i,a]}$ for each $(i,a) \in \II \times \C^\times$.
Clearly this is a surjection, but not an injection as $\phi^\sigma_\fg(Y_{i,a}) = \phi^\sigma_\fg(Y_{j,b})$ if (and only if) $[i,a] = [j,b]$.
Note that the map $\phi^\sigma_\fg$ sends a dominant monomial to a $\sigma$-invariant dominant monomial, and satisfies $\phi^\sigma_\fg(A_{i,a}) = A_{[i,a]}$ for each $(i,a) \in \II \times \C^\times$,
where $A_{i,a}$ is defined as in \eqref{eq:defA} for $\fg$ (of simply-laced type).

\begin{thm}[{\cite[Theorem 4.15]{Her10}}] \label{Thm:Her10}
There is a unique ring homomorphism 
\[\Phi^\sigma_\fg \colon K(\Cc) \to K(\Cc^\sigma)\]
making the following diagram commute:
\[ 
\xymatrix{ 
K(\Cc) \ar@{^(->}[r]^-{\chi_q} \ar[d]_-{\Phi^\sigma_\fg} & \Z[Y_{i,a}^{\pm 1} \mid (i,a) \in \II \times \C^\times ] \ar[d]^-{\phi^\sigma_\fg} \\
K(\Cc^\sigma) \ar@{^(->}[r]^-{\chi_q^\sigma} & \Z[Y_{[i,a]}^{\pm 1} \mid [i,a] \in (\II \times \C^\times)/\sigma ].
}
\]
Moreover, we have the equality
\begin{equation}\label{eq:folding}
 \Phi^\sigma_\fg[L(m)] = [L^\sigma(\phi_\fg^\sigma(m))], \quad \text{or equivalently,} \quad \phi^\sigma_\fg(\chi_q(L(m)))=
\chi_q^\sigma(L^\sigma(\phi^\sigma_\fg(m)))
\end{equation}
if $L(m)$ is a Kirillov--Reshetikhin module, i.e., if $m = \prod_{s=0}^k Y_{i,aq^{2s}}$ for some $i \in \II$, $a \in \C^\times$ and $k \in \Z_{\ge 0}$.
\end{thm}

{
\begin{rem}\label{rem:fold}
One may expect that the equalities \eqref{eq:folding} hold for an arbitrary dominant monomial $m$.
This is false in general: \cite[Example 2.21]{Wang} provides a counterexample. See Conjecture \ref{Conj:tw} below for a conjectural setting in which these equalities hold.
\end{rem}
}

\subsection{Twisted analog of the skeletal subcategory} \label{Ssec:twHLcat}
Hereafter, as in \S\ref{Ssec:HLcat}, we set
\[ Y_{i,p}\seq Y_{i,q^p}, \qquad Y_{[i,p]} \seq Y_{[i,q^p]}, \qquad A_{i,p} \seq A_{i,q^p}, \qquad  A_{[i,p]} \seq A_{[i,q^p]}\]
for each $(i,p) \in \II \times \Z$.
Note that the map $\II \times \Z \ni (i,p) \mapsto [i,q^p] \in (\II \times \C^\times)/\sigma$ is injective as $q$ is not a root of unity, and hence the set $\{ Y_{[i,p]} \mid (i,p) \in \II \times \Z\}$ is algebraically independent.
Consider the rings of Laurent polynomials: 
\[ \cY \seq \Z[Y_{i,p}^{\pm 1} \mid i \in \II, p \in \Z ], \qquad  \cY^\sigma \seq \Z[Y_{[i,p]}^{\pm 1} \mid i \in \II, p \in \Z ].\]
Let  $\cM$ (resp.\ $\cM^\sigma$) denote the subset of $\cY$ (resp.\ $\cY^\sigma$) consisting of all the Laurent monomials in $Y_{i,p}$'s (resp.\ $Y_{[i,p]}$'s), and $\cM_+$ (resp.\ $\cM^\sigma_{+}$) its subset consisting of dominant monomials. 
Analogously to the case of $A_{i,p}$, the element $A_{[i,p]}$ belongs to $\cM^\sigma$ for each $(i,p) \in \II \times \Z$.
For $m, m' \in \cM^\sigma$, we have $m' \le^\sigma m$ if and only if $m'm^{-1}$ is a monomial in  these $A_{[i,p]}^{-1}$'s. 

As the twisted analog of the category $\Cc_\Z$ in \S\ref{Ssec:HLcat}, 
we define the category $\Cc^\sigma_{\Z}$ to be the Serre subcategory of $\Cc^\sigma$ generated by the simple modules $L^\sigma(m)$ with $m \in \cM^\sigma_+$.
The category $\Cc^\sigma_\Z$ is closed under the tensor product (see the first paragraph of the proof of \cite[Theorem 4.15]{Her10}). 
Moreover, it can be thought of as a monoidal skeleton of the category $\Cc^\sigma$. In fact, any simple module in $\Cc^\sigma$ factorizes into a tensor product of some spectral parameter shifts of simple modules in $\Cc^\sigma_\Z$. 
(This follows from the fact that the category $\Cc^\sigma_\Z$ contains the Hernandez--Leclerc category in the sense of \cite[\S2.3]{KKOPblock}.) 

\begin{prop}[\cite{Her10}] \label{Prop:CZtw}
The twisted $q$-character map restricts to an injective ring homomorphism
\[ \chi_q^\sigma \colon K(\Cc^\sigma_\Z) \hookrightarrow \cY^\sigma.\]
\end{prop}
\begin{proof}
See again the first paragraph of the proof of \cite[Theorem 4.15]{Her10}.
\end{proof}

The folding homomorphism $\phi^\sigma_\fg$ restricts to the isomorphism $\cY \xrightarrow{\simeq} \cY^\sigma$, mapping $Y_{i,p}$ to $Y_{[i,p]}$ for each $(i,p) \in \II \times \Z$,
which we denote again by $\phi^\sigma_\fg$.
In particular, it induces the isomorphisms $\cM \cong \cM^\sigma$ and $\cM_+ \cong \cM^\sigma_+$, and satisfies
\begin{equation} \label{eq:foldA}
\phi^\sigma_\fg(A_{i,p}) = A_{[i,p]} \qquad \text{for each $(i,p) \in \II \times \Z$}.
\end{equation}
The following statement is a combination of Theorem \ref{Thm:Her10} and Proposition \ref{Prop:CZtw}.
\begin{prop}[cf.\ {\cite[Theorem 4.15]{Her10}}] \label{Prop:foldCZ}
The folding homomorphism $\Phi^\sigma_\fg$ restricts to the ring isomorphism
\[ \Phi^\sigma_\fg \colon K(\Cc_\Z) \xrightarrow{\simeq} K(\Cc^\sigma_\Z)\]
fitting into the following commutative diagram:
\[
\xymatrix{ 
K(\Cc_\Z) \ar@{^(->}[r]^-{\chi_q} \ar[d]_-{\Phi^\sigma_\fg}^-{\simeq} & \cY  \ar[d]^-{\phi^\sigma_\fg}_-{\simeq} \\
K(\Cc^\sigma_\Z) \ar@{^(->}[r]^-{\chi_q^\sigma} & \cY^\sigma.
}
\]
\end{prop}

{
\begin{conj}[{\cite[Conjecture 2.20]{Wang}}] \label{Conj:tw}
The isomorphism $\Phi^\sigma_\fg$ in Proposition \ref{Prop:foldCZ} induces a bijection between the simple isomorphism classes. 
In other words, the equalities \eqref{eq:folding} hold for any dominant monomial $m \in \cM_+$.
\end{conj}

Note that, if Conjecture \ref{Conj:tw} is true, it means that one can compute every simple twisted $q$-character from the simple (untwisted) $q$-characters. 
Thus, the computation problem of simple twisted $q$-characters reduces to that of simple (untwisted) $q$-characters.  

When $\fg$ is of type $\mathrm{A}$, Conjecture \ref{Conj:tw} was previously verified by Kang--Kashiwara--Kim--Oh in \cite{KKKOsw3} (see also \S\ref{Ssec:fold-qchar} below), and
we will prove it for doubly-twisted quantum loop algebras of type $\mathrm{AD}$ in the sequel.
}

\subsection{Folding simple $q$-characters} \label{Ssec:fold-qchar}

Now we state our main result in this section.

\begin{thm}\label{Thm:tw-main}
If $(\mathrm{X}_n,r) = (\mathrm{A}_n,2)$ or $(\mathrm{D}_n,2)$, we have
the equalities \eqref{eq:folding} for any $m \in \cM_+$. 
In particular, the bijection $[L(m)] \mapsto [L^\sigma(\phi^\sigma_\fg(m))]$ between the simple isomorphism classes of $\Cc_\Z$ and those of $\Cc^\sigma_\Z$ preserves the dimension of modules.
\end{thm}


A proof of Theorem \ref{Thm:tw-main} is given in \S\ref{Ssec:prf-tw-main} below. 
Note that the case of $(\mathrm{X}_n, r) = (\mathrm{A}_n, 2)$ is already known by Kang--Kashiwara--Kim--Oh as the main result of \cite{KKKOsw3}. 
In this paper, we give an alternative proof.
The case of $(\mathrm{X}_n, r) = (\mathrm{D}_n, 2)$ is new as far as the authors know.
Our proof utilizes a twisted analog of the freezing operator, which we explain in the next subsection. In addition, the following result due to Kang--Kashiwara--Kim--Oh \cite{KKKOsw4} (when $(\mathrm{X}_n, r) = (\mathrm{A}_n,2), (\mathrm{D}_n,2)$) will play a key role. 
A similar result is obtained when $(\mathrm{X}_n,r) = (\mathrm{E}_6,2), (\mathrm{D}_4, 3)$ as well by Oh--Scrimshaw  \cite{OS}.

Recall the subset $\cM[a,b]_+ \subset \cM_+$ associated with a finite integer interval $[a,b] \subset \Z$ and the corresponding monoidal subcategory $\Cc_{[a,b]} \subset \Cc_\Z$ defined in \S\ref{Ssec:Hconj} (before Theorem \ref{thm:H-conj-region-Q}).
We let $\cM[a,b]^\sigma_+ \seq \phi^\sigma_\fg(\cM[a,b]_+)$ and define the category $\Cc_{[a,b]}^\sigma$ to be the Serre subcategory of $\Cc^\sigma_\Z$ generated by the simple representations $L^\sigma(m)$ with $m \in \cM[a,b]^\sigma_+$.
This is a monoidal subcategory of $\Cc_{\Z}^\sigma$.
The folding homomorphism $\Phi^\sigma_\fg$ restricts to the ring isomorphism
\[ \Phi^\sigma_\fg \colon K(\Cc_{[a,b]}) \to K(\Cc^\sigma_{[a,b]})\]
for any finite integer interval $[a,b] \subset \Z$.  

\begin{thm}[{\cite{KKKOsw4, OS}}] \label{thm:tw-sw}
For any $(\fg,\sigma)$ as in \S\ref{Ssec:tw-setup}, the equalities \eqref{eq:folding} hold if $m \in \cM[a,b]_+$ for some finite integer interval $[a,b] \subset \Z$ satisfying $b-a < h^\vee$, where $h^\vee$ denotes the (dual) Coxeter number of $\fg$. 
\end{thm}
\begin{proof}
The result follows from \cite[Corollary 5.6]{KKKOsw4} and \cite[Theorem 6.9(4)]{OS}. 
Indeed, we observe that the Grothendieck ring $K(\Cc_{[a,a+h^\vee-1]})$ (resp.\ $K(\Cc_{[a,a+h^\vee-1]}^{\sigma})$) is identical (up to spectral parameter shift) to the tensor product of rings $K(\mathcal{C}_Q^{(1)}) \otimes_\Z K(\mathcal{C}_{Q^\mathrm{op}}^{(1)})$ (resp.\ $K(\mathcal{C}_Q^{(r)}) \otimes_\Z K(\mathcal{C}_{Q^\mathrm{op}}^{(r)})$), where $\mathcal{C}_Q^{(1)}$ (resp.\ $\mathcal{C}_Q^{(r)}$) denotes the core subcategory associated with a sink-source orientation $Q$ of the Dynkin diagram of $\fg$ and $\mathcal{C}_{Q^\mathrm{op}}^{(1)}$ (resp.\ $\mathcal{C}_{Q^{\mathrm{op}}}^{(r)}$) the one associated with the opposite orientation $Q^{\mathrm{op}}$. 
Since we know from Theorem \ref{Thm:Her10} that the folding isomorphism $\Phi^\sigma_\fg$ sends $[L(Y_{i,p})]$ to $[L^\sigma(Y_{[i,p]})]$ for any $(i,p) \in I \times \Z$, it is immediate to check that the restriction $\Phi^\sigma_\fg \colon K(\Cc_{[a,a+h^\vee-1]}) \to K(\Cc^\sigma_{[a,a+h^\vee-1]})$ coincides with the tensor product of the isomorphisms $[\mathcal{F}_Q^{(r)}]\circ[\mathcal{F}_Q^{(1)}]^{-1} \colon K(\mathcal{C}_Q^{(1)}) \to K(\mathcal{C}_Q^{(r)})$ and $[\mathcal{F}_{Q^{\mathrm{op}}}^{(r)}]\circ[\mathcal{F}_{Q^{\mathrm{op}}}^{(1)}]^{-1} \colon K(\mathcal{C}_{Q^{\mathrm{op}}}^{(1)}) \to K(\mathcal{C}_{Q^{\mathrm{op}}}^{(r)})$ arising from the generalized quantum affine Schur--Weyl duality functors $\mathcal{F}_Q^{(1)}$, $\mathcal{F}_Q^{(r)}$, $\mathcal{F}_{Q^{\mathrm{op}}}^{(1)}$ and $\mathcal{F}_{Q^{\mathrm{op}}}^{(r)}$ (with the notation in \cite{KKKOsw4, OS}). 
\end{proof}

\subsection{Freezing twisted $q$-characters}
Let $(\fg,\sigma)$ be the pair of a complex finite-dimensional simple Lie algebra $\fg$ of type $\mathrm{X}_n$ with $\mathrm{X} \in \{\mathrm{A,D,E}\}$ and a non-trivial automorphism $\sigma$ of the Dynkin diagram of $\fg$, as in \S\ref{Ssec:tw-setup}.
We consider another such pair $(\tfg, \tsigma)$ and assume that there is an inclusion $\II \subset \tII$ of the sets of nodes of Dynkin diagrams satisfying $\tc_{ij} = c_{ij}$ for any $i,j \in \II$, where
$(c_{ij})_{i,j \in \II}$ and $(\tc_{ij})_{i,j \in \tII}$ are the Cartan matrices of $\fg$ and $\tfg$ respectively.
We further assume that the automorphism $\tsigma$ preserves the subset $\II$ and the restriction $\tsigma|_\II$ coincides with $\sigma$.
By the classification of $(\fg,\sigma)$ in \S\ref{Ssec:tw-setup}, such a situation is limited. 
In particular, $\tsigma$ is uniquely determined by the inclusion $\II \subset \tII$ and $\sigma$, and we have $\ord(\tsigma) = \ord(\sigma)=r$.
Therefore, by an abuse of notation, we shall simply write $\sigma$ instead of $\tsigma$ in what follows.

\begin{eg}\label{eg:tw-incl}
In this paper, we are mainly interested in the case when $\fg$ is of type $\mathrm{X}_n$ and $\tfg$ is of type $\mathrm{X}_{\tilde{n}}$ with $\mathrm{X} \in \{\mathrm{A,D}\}$ and $\tilde{n} \ge n$, and $\sigma$ is of order $r=2$. 
Here we assume that $\tilde{n}$ has the same parity as $n$ if $\mathrm{X}=\mathrm{A}$.
In this case, we make the identification
\[
\II = \begin{cases}
[-l,l] \setminus \{0\} & \text{if $\mathrm{X}_n = \mathrm{A}_{2l}$}, \\
[-l,l] & \text{if $\mathrm{X}_n = \mathrm{A}_{2l+1}$}, \\
[1,n] & \text{if $\mathrm{X}_n = \mathrm{D}_n$},
\end{cases}
\]
so that the nodes of Dynkin diagram of $\fg$ are labelled as depicted in \S\ref{Ssec:tw-setup}. 
We use the analogous identification for $\tII$ as well.
Then, the natural inclusion $\II \subset \tII$ satisfies the above assumptions.
\end{eg}

Retain the notations $\tcY$, $\tcM$, $\tcM_+$, $\tilde{Y}_{i,p}$, $\tilde{A}_{i,p}$, $\tilde{L}(m)$ for $\tqlp$ from \S\ref{Sec:frz}.
Let $\tcY^{\sigma}$, $\tcM^{\sigma}$, $\tcM^{\sigma}_+$, $\tilde{Y}_{[i,p]}$, $\tilde{A}_{[i,p]}$, $\tilde{L}^\sigma(m)$ for their respective counterparts for $\tqlps$, defined as in the previous subsections.   
We have defined the partial ordering $\le^{\sigma}$ of the set $\tcM^{\sigma}$ with the variables $\tilde{A}_{[i,p]}$. 
As in the untwisted case, we define another partial ordering $\le^{\sigma}_\II$ associated with the inclusion $\II \subset \tII$ so that we have $m' \le_\II^{\sigma} m$ if and only if $m'm^{-1}$ is a monomial in the variables $\tilde{A}_{[i,p]}^{-1}$ for $i \in \II$ and $p \in \Z$.

Consider the ring homomorphism $\res^\sigma_\II \colon \tcY^\sigma \to \cY^\sigma$ given by 
\[ \res_{\II}^\sigma(\tilde{Y}_{[i,p]}) \seq \begin{cases}
Y_{[i,p]} & \text {if $i \in \II$},\\
1 & \text{if $i \not \in \II$}
\end{cases} \]
for $(i,p) \in \tII \times \Z$.
We have $\res_{\II}^\sigma(\tilde{A}_{[i,p]}) = A_{[i,p]}$ for any $(i,p) \in \II \times \Z$. 
Therefore, the following implication is true for $m,m' \in \tcM^\sigma$:
\begin{equation} \label{eq:tw-imply}
m \le^\sigma_\II m' \quad  \Longrightarrow \quad \res^\sigma_\II(m) \le^\sigma \res^\sigma_\II(m').
\end{equation} 
Moreover, it is obvious from the definition that the following diagram commutes:
\begin{equation}
\label{eq:res-fold}
\vcenter{\xymatrix{ 
\tcY \ar[r]^-{\res_\II}\ar[d]_-{\phi^\sigma_{\tfg}} & \cY \ar[d]^-{\phi^{\sigma}_{\fg}}\\
\tcY^{\sigma} \ar[r]^-{\res^\sigma_\II} & \cY^{\sigma}.
}}
\end{equation}

As in the untwisted case, 
we say that an element $y$ of $\cY^\sigma$ is $m$-pointed, for $m \in \cM^\sigma$, if it is written in the form
\[
y= m + \sum_{m' \in \cM^\sigma, m' <^\sigma m} c_{m'}m'
\]
for some $c_{m'}\in \Z$.
For any $m \in \cM^\sigma_+$, the simple twisted $q$-character $\chi_q^\sigma(L^\sigma(m))$ is $m$-pointed by Proposition \ref{prop:tw-domin}.
An element of $\cY^\sigma$ is called pointed if it is $m$-pointed for some $m \in \cM^\sigma$.
Let $\cY^{\sigma,\ptd}$ denote the set of pointed elements in $\cY^\sigma$.
Similarly, we define $\tcY^{\sigma,\ptd} \subset \tcY^\sigma$.

Analogously to the untwisted case, for any $m$-pointed element $y$ of $\tcY^\sigma$ (with $m \in \tcM^\sigma$)  written as
\[ y= m + \sum_{m' \in \tcM^\sigma, m' <^\sigma m} c_{m'}m'\]
with $c_{m'} \in \Z$, we define its freezing $\frz^{\tfg, \sigma}_{\fg}(y)$ to be the following element of $\cY^\sigma$:
\[ \frz^{\tfg, \sigma}_{\fg}(y) \seq \res^\sigma_\II \left( m + \sum_{m' \in \tcM^\sigma, m' <^\sigma_\II m} c_{m'}m' \right).\]
By \eqref{eq:tw-imply}, this is $\res^\sigma_\II(m)$-pointed. 
Thus, the assignment $y \mapsto \frz^{\fg, \sigma}_{\tfg}(y)$ defines the map $\frz^{\fg, \sigma}_{\tfg} \colon\tcY^{\sigma,\ptd} \to \cY^{\sigma,\ptd}$.

Note that the folding homomorphisms restrict to the bijections
\[ \phi^\sigma_\fg \colon \cY^\ptd \to \cY^{\sigma,\ptd}, \qquad \phi^{\sigma}_{\tfg} \colon \tcY^{\ptd} \to \tcY^{\sigma, \ptd}\]
thanks to \eqref{eq:foldA} and its $\tfg$-analog.
It is immediate that the following diagram commutes:
\begin{equation} \label{eq:frzg-fold}
\vcenter{\xymatrix{ 
\tcY^\ptd \ar[r]^-{\frz^{\tfg}_{\fg}}\ar[d]_-{\phi^\sigma_{\tfg}} & \cY^\ptd \ar[d]^-{\phi^{\sigma}_{\fg}}\\
\tcY^{\sigma, \ptd} \ar[r]^-{\frz^{\tfg, \sigma}_{\fg}} & \cY^{\sigma,\ptd}.
}}
\end{equation}

The next statement is the twisted analog of Proposition \ref{prop:frz-chiq}, whose proof we omit as it is parallel to the untwisted case. (Note that we have the twisted analog of \cite[Lemma 5.9]{Her08} by the same proof. For example, a special case when $I$ consists of a single $\sigma$-orbit is stated in \cite[Proposition 4.13]{Her10}.) 
\begin{prop} \label{prop:tw-frz-chiq}
For any $m \in \tcM^\sigma_+$, we have
\begin{equation}\label{eq:frz-chiq-g}
\frz^{\tfg, \sigma}_{\fg}(\chi_q^\sigma(\tilde{L}^\sigma(m))) = \chi_q^{\sigma}(L^{\sigma}(\res_\II^\sigma(m))). 
\end{equation}
\end{prop}

\subsection{Proof of Theorem \ref{Thm:tw-main}}\label{Ssec:prf-tw-main}
Now, we are ready to prove Theorem \ref{Thm:tw-main}.
Let $n$ and $\tilde{n}$ be integers satisfying $1\le n<\tilde{n}$.
Let $\fg$ and $\tfg$ be complex finite-dimensional simple Lie algebras of type $\mathrm{X}_n$ and $\mathrm{X}_{\tilde{n}}$ respectively, with $\mathrm{X} \in \{\mathrm{A,D} \}$.
We choose the inclusion $\II \subset \tII$ as in Example \ref{eg:tw-incl}.
After relabeling if necessary, let $\sigma \colon \II \to \II$ be the unique non-trivial diagram automorphism of order $r = 2$ in Section \ref{Ssec:tw-setup}, which extends to a diagram automorphism on $\tilde{I}$ in the unique way.
We fix $n$ and choose $\tilde{n}$ sufficiently large; the precise bound (which depends on the simple module in question) will be given below.

Given an arbitrary dominant monomial $m \in \cM_+$, our goal is to show the equality
\begin{equation}\label{eq:tw-goal}
\phi^{\sigma}_\fg(\chi_q(L(m))) = \chi_q^{\sigma}(L^{\sigma}(\phi^\sigma_\fg(m))). 
\end{equation}
Choose a finite integer interval $[a,b] \subset \Z$ large enough so that we have $m \in \cM[a,b]_+$.
Depending on this choice of interval $[a,b]$ (and hence on the given $m$), we choose the integer $\tilde{n}$ large enough so that $b-a$ is smaller than the (dual) Coxeter number $\tilde{h}^\vee$ of $\tfg$.  
More explicitly, we choose $\tilde{n} \in \Z_{> n}$ so that $\tilde{n}+1 > b-a$ holds when $\mathrm{X = A}$, and $2(\tilde{n}-1) > b-a$ holds when $\mathrm{X=D}$. 
Let $\tilde{m}\in\tcM[a,b]_+$ be the dominant monomial obtained from $m$ by replacing each factor $Y_{i,p}$ with $\tilde{Y}_{i,p}$. 
Obviously, we have $\res_\II(\tilde{m})=m$. For the corresponding simple $\tqlp$-module $\tL(\tilde{m})$, we can apply Theorem \ref{thm:tw-sw} to get
\begin{equation} \label{eq:tw-KL-large-n}
\phi^\sigma_{\tfg} (\chi_q(\tL(\tilde{m}))) = \chi_q^\sigma(\tL^\sigma(\phi^\sigma_{\tfg}(\tilde{m}))).
\end{equation} 
Then, we obtain
\begin{align*}
\phi^{\sigma}_{\fg}(\chi_q(L(m)))&= \phi^{\sigma}_{\fg}(\frz^{\tfg}_{\fg}(\chi_q(\tL(\tilde{m}))))&& \text{by Proposition \ref{prop:frz-chiq} and $\res_\II(\tilde{m})=m$,} \\
&= \frz^{\tfg, \sigma}_{\fg}( \phi^{\sigma}_{\tfg}(\chi_q(\tL(\tilde{m}))))&& \text{by the commutativity of \eqref{eq:frzg-fold},} \\  
&= \frz^{\tfg, \sigma}_{\fg}(\chi_q^\sigma(\tL^\sigma(\phi^\sigma_{\tfg}(\tilde{m})))) && \text{by \eqref{eq:tw-KL-large-n},} \\
&= \chi_q^{\sigma}(L^{\sigma}(\res_\II^\sigma(\phi^\sigma_{\tfg}(\tilde{m})))) && \text{by Proposition \ref{prop:tw-frz-chiq},} \\
&= \chi_q^{\sigma}(L^{\sigma}(\phi^\sigma_{\fg}(m))) && \text{by the commutativity of \eqref{eq:res-fold} and $\res_\II(\tilde{m})=m$,}
\end{align*}
which verifies the desired equality \eqref{eq:tw-goal}.

    
    \def\cprime{$'$}
\providecommand{\bysame}{\leavevmode\hbox to3em{\hrulefill}\thinspace}
\providecommand{\MR}{\relax\ifhmode\unskip\space\fi MR }
\providecommand{\MRhref}[2]{%
  \href{http://www.ams.org/mathscinet-getitem?mr=#1}{#2}
}
\providecommand{\href}[2]{#2}



\end{document}

%% file: macroEnv.tex
\ifdefined\thm
\else
	\theoremstyle{plain}
 	\newtheorem{thm}{Theorem}[section]
\fi

\ifdefined\prop
\else
	\theoremstyle{plain}
 	\newtheorem{prop}[thm]{Proposition}
\fi\ifdefined\cor
\else
	\theoremstyle{plain}
 	\newtheorem{cor}[thm]{Corollary}
\fi

\ifdefined\lem
\else
	\theoremstyle{plain}
 	\newtheorem{lem}[thm]{Lemma}
\fi

\ifdefined\def
\else
	\theoremstyle{definition}
	\newtheorem{def}[thm]{Definition}
\fi

\ifdefined\defn
\else
	\theoremstyle{definition}
	\newtheorem{defn}[thm]{Definition}
\fi

\ifdefined\example
\else
	\theoremstyle{definition}
    \newtheorem{example}[thm]{Example}
\fi

\ifdefined\eg
\else
		\theoremstyle{definition}
        \newtheorem{eg}[thm]{Example}
\fi

\ifdefined\rem
\else
		\theoremstyle{remark}
        \newtheorem{rem}[thm]{Remark}
\fi

%% file: macro.tex
\usepackage{tikz-cd}
\usepackage{stackrel}

\usepackage{graphicx}

\usepackage{tikz}
\usetikzlibrary{positioning,shapes,shadows,arrows,snakes,arrows.meta}

\pgfdeclarelayer{edgelayer}
\pgfdeclarelayer{nodelayer}
\pgfsetlayers{edgelayer,nodelayer,main}

\tikzstyle{none}=[inner sep=0pt]
\tikzstyle{black box}=[draw=black, fill=black!25]
\tikzstyle{white box}=[draw=black, fill=white]
\tikzstyle{black circle}=[circle,draw=black!50, fill=black!25]
\tikzstyle{red circle}=[circle,draw=red!50, fill=red!25]
\tikzstyle{blue circle}=[circle,draw=blue!50, fill=blue!25]
\tikzstyle{green circle}=[circle,draw=green!50, fill=green!25]
\tikzstyle{yellow circle}=[circle,draw=yellow!50, fill=yellow!25]

\tikzstyle{unfrozen}=[circle,inner sep=1pt,minimum size=1pt,draw=black!100,fill=red!50]
\tikzstyle{frozen}=[rectangle,inner sep=1pt,minimum size=12pt,draw=black!75,fill=cyan!50]
\tikzstyle{boundary}=[-,draw=cyan]
\tikzstyle{internal}=[-,draw=red]
\tikzstyle{pre}=[<-,shorten <=1pt,>={Stealth[round]},semithick]
\tikzstyle{post}=[->,shorten >=1pt,>={Stealth[round]},semithick]
\tikzstyle{point}= [circle,inner sep=0pt, outer sep=1.5pt,minimum size=1.5pt,draw=black!100,fill=black!100]
\tikzstyle{every label}= [black]

\newcommand{\LP}{\mathcal{LP}}

\newcommand{\frz}{\mathfrak{f}}

\newcommand{\cC}{\mathcal{C}}

\newcommand{\hf}{\frac{1}{2}}

\renewcommand{\tw}{\operatorname{tw}}

\ifdefined\supp
\else
\newcommand{\supp}{\operatorname{supp}}
\fi

\ifdefined\tw
\else
	\newcommand{\tw}{{\opname tw}}
\fi

\ifdefined\G
    \renewcommand{\G}{{\mathbb G}}
\else
   \newcommand{\G}{{\mathbb G}}
\fi

\ifdefined\C
    \renewcommand{\C}{{\mathbb C}}
\else
   \newcommand{\C}{{\mathbb C}}
\fi

\usepackage{comment}
